\documentclass[a4paper]{amsart}

\usepackage{diagrams, mathrsfs, amssymb, mathtools, fullpage, enumerate}
\usepackage[pagebackref]{hyperref}

\renewcommand{\AA}{\mathbf{A}}
\newcommand{\BB}{\mathbf{B}}
\newcommand{\CC}{\mathbf{C}}
\newcommand{\DD}{\mathbf{D}}
\newcommand{\FF}{\mathbf{F}}
\newcommand{\RR}{\mathbf{R}}
\newcommand{\QQ}{\mathbf{Q}}
\newcommand{\ZZ}{\mathbf{Z}}

\newcommand{\QQbar}{\overline{\QQ}}
\newcommand{\Qpbar}{\QQbar_p}
\newcommand{\Qp}{\QQ_p}
\newcommand{\Zp}{\ZZ_p}
\newcommand{\Qpi}{\QQ_{p, \infty}}
\newcommand{\Qi}{\QQ_{\infty}}

\newcommand{\sH}{\mathscr{H}}
\newcommand{\sF}{\mathscr{F}}

\newcommand{\cO}{\mathcal{O}}
\newcommand{\cH}{\mathcal{H}}
\newcommand{\cF}{\mathcal{F}}
\newcommand{\cP}{\mathcal{P}}
\newcommand{\cW}{\mathcal{W}}
\newcommand{\cR}{\mathcal{R}}

\newcommand{\et}{\text{\textup{\'et}}}
\newcommand{\dR}{\mathrm{dR}}
\newcommand{\Iw}{\mathrm{Iw}}
\newcommand{\cris}{\mathrm{cris}}

\newcommand{\Dcris}{\DD_{\cris}}
\newcommand{\Hf}{H^1_{\mathrm{f}}}

\newcommand{\frP}{\mathfrak{P}}

\newcommand{\frm}{\mathfrak{m}}

\newcommand{\BF}{\mathcal{BF}}
\newcommand{\cBF}{{}_c\BF}

\newcommand{\into}{\hookrightarrow}

\newcommand{\imp}{\mathrm{imp}}

\newcommand{\bT}{\mathbf{T}}
\newcommand{\bk}{\mathbf{k}}

\DeclareMathOperator{\Gr}{Gr}
\DeclareMathOperator{\Sel}{Sel}

\DeclareMathOperator{\loc}{loc}

\DeclareMathOperator{\Frob}{Frob}
\DeclareMathOperator{\Spec}{Spec}

\DeclareMathOperator{\Hom}{Hom}
\DeclareMathOperator{\Fil}{Fil}
\DeclareMathOperator{\Sym}{Sym}
\DeclareMathOperator{\TSym}{TSym}
\DeclareMathOperator{\Gal}{Gal}

\DeclareMathOperator{\Aut}{Aut}

\DeclareMathOperator{\Char}{char}

\DeclareMathOperator{\SL}{SL}
\DeclareMathOperator{\GL}{GL}

\title[Iwasawa theory for the symmetric square]{Iwasawa theory for the symmetric square of a modular form}

\author{David Loeffler}
\address{Mathematics Institute, University of Warwick, Coventry CV4 7AL, UK}
\email{d.a.loeffler@warwick.ac.uk}

\author{Sarah Livia Zerbes}
\address{Department of Mathematics, University College London, London WC1E 6BT, UK}
\email{s.zerbes@ucl.ac.uk}

\thanks{The authors' research was supported by the following grants: Royal Society University Research Fellowship (Loeffler); ERC Consolidator Grant ``Euler Systems and the Birch--Swinnerton-Dyer conjecture'' (Zerbes).}

\theoremstyle{plain}

\newtheorem{theorem}{Theorem}[subsection]
\newtheorem{lemma}[theorem]{Lemma}
\newtheorem{proposition}[theorem]{Proposition}
\newtheorem{corollary}[theorem]{Corollary}

\newtheorem{lettertheorem}{Theorem}

\theoremstyle{definition}
\newtheorem{definition}[theorem]{Definition}
\newtheorem{notation}[theorem]{Notation}
\newtheorem{conjecture}[theorem]{Conjecture}

\theoremstyle{remark}
\newtheorem{remark}[theorem]{Remark}
\newtheorem{note}[theorem]{Note}

\begin{document}

\begin{abstract}
 We construct a compatible family of global cohomology classes (an Euler system) for the symmetric square of a modular form, and apply this to bounding Selmer groups of the symmetric square Galois representation and its twists.
\end{abstract}

\maketitle

 \begin{center}
\emph{To the memory of Rudolf Zerbes (1944---2015)}
\end{center}
\vspace{1ex}
 
 \section{Introduction}
 
  \subsection{Background}
  
   In this paper, we shall study the arithmetic of a specific class of $p$-adic Galois representations: those associated to the symmetric squares of modular forms. These Galois representations are important for two reasons. Firstly, they are very natural examples of global Galois representations, and thus provide a good testing ground for general conjectures relating arithmetic objects to values of $L$-functions. Secondly, there is a strong connection between the deformation theory of the 2-dimensional standard representation of a modular form, and the Galois cohomology of its 3-dimensional symmetric square representation; hence, these Galois cohomology groups are of central importance in Wiles' work on Fermat's last theorem \cite{wiles95} and the subject of modularity lifting which grew from it.
   
   For simplicity, in this Introduction, we shall only explain our main results (Theorems A, B and C below) in the technically simplest case in which $F$ is a normalised eigenform
   of level $1$ and weight $k$. Let $p$ be a prime; we assume throughout this paper that $p\geq 7$. We make the following hypothesis of ordinarity of $F$ at $p$: we choose a prime of the coefficient field of $F$ above $p$, and we assume that the $p$-th Fourier coefficient $a_p(F)$ is a unit at this prime. In this introduction, we shall make the further simplifying hypothesis that our chosen prime has degree 1, and thus defines an embedding of the coefficient field of $F$ into $\QQ_p$.

   We write $L(\Sym^2(F), s)$ for the complex $L$-function of the symmetric square of the
   Galois representation attached to $F$, which is known to be entire by work of Shimura \cite{shimura75}
   and Gelbart-Jacquet \cite{gelbartjacquet78}. Thanks to our ordinarity
   hypothesis and some classical algebraicity results for the critical values of $L(\Sym^2(F),\psi,s)$ where $\psi$ denotes an arbitrary Dirichlet character, this complex $L$-function has a $p$-adic
   analogue, which we will denote by $L_p(\Sym^2(F))$, which is an element of the Iwasawa
   algebra of the Galois group $\mathcal{G}$ of the field $\QQ(\mu_{p^\infty})$ over $\QQ$, where $\mu_{p^\infty}$ denotes the group
   of all $p$-power roots of unity. One is interested in this $p$-adic analogue because, unlike its
   complex counterpart, it is conjecturally linked to Galois cohomology via a main conjecture.
   Indeed, if we write $M_{\Zp}(F)$ for a Galois stable lattice in the
   $p$-adic Galois representation attached to $F$, and $\Sym^2M_{\Zp}(F)$ for its symmetric square,
   then one can define an analogue of the Selmer group   
   \[ \Sel\left(\QQ(\mu_{p^\infty}),\Sym^2M_{\Zp}(F)\right)\subset H^1\left(\QQ(\mu_{p^\infty}),\Sym^2M_{\Zp}(F)\right).\]
   This Selmer group has a natural structure as a module over the Iwasawa algebra of $\mathcal{G}$, and
   it can be shown that its Pontrjagin dual, which we denote by $\operatorname{Sel}\left(\QQ(\mu_{p^\infty}), \Sym^2 M_{\Zp}(F) \right)^\vee$,
   is a finitely generated torsion module over the Iwasawa algebra of $\mathcal{G}$. Since $p$ is odd,
   this means that we can define the characteristic ideal of $\operatorname{Sel}\left(\QQ(\mu_{p^\infty}), \Sym^2 M_{\Zp}(F) \right)^\vee$ in the
   Iwasawa algebra $\Lambda$ of $\mathcal{G}$. Writing $\operatorname{char}_{\Lambda} \operatorname{Sel}\left(\QQ(\mu_{p^\infty}), \Sym^2 M_{\Zp}(F) \right)^\vee $ for this characteristic ideal, the main conjecture would then be the assertion that   
   \begin{equation}
    \label{eq:mainconj}
    \operatorname{char}_{\Lambda} \operatorname{Sel}\left(\QQ(\mu_{p^\infty}), \Sym^2 M_{\Zp}(F) \right)^\vee = \left( L_p(\Sym^2 F) \right)
   \end{equation}
   subject to a good choice of the appropriate $p$-adic periods which enter into the definition
   of $L_p(\Sym^2(F))$. More generally, there is an analogue of this main conjecture with
   the Galois representation $\Sym^2M_{\Zp}(F)$ replaced by its tensor product with a Dirchlet
   character of conductor prime to $p$, and, for simplicity, having values in $\ZZ_p$. 
   
   One strategy for attacking Iwasawa main conjectures, which has proved very successful in other contexts (e.g.~\cite{rubin91, kato04, KLZ1b}), is the construction of an \emph{Euler system}: a family of Galois cohomology classes over cyclotomic extensions, satisfying compatibility properties under the norm map. When a sufficiently non-trivial Euler system exists, it can often be used to prove one inclusion in the Main Conjecture: that the $p$-adic $L$-function lies in the characteristic ideal of the corresponding Selmer group. The work of Flach \cite{flach92} can be interpreted as a ``fragment'' of an Euler system for the symmetric square of an elliptic curve; however, despite strenuous efforts, no construction of a full Euler system based on Flach's methods has so far been successful. 
   
   This absence of an Euler system was, of course, the primary motivation for the development of another major tool for studying Galois cohomology: the \emph{Taylor--Wiles method}, introduced in \cite{taylorwiles95}. This method gives extremely precise descriptions of the Selmer group of a certain Tate twist of $\Sym^2 F$ over $\QQ$ (rather than $\QQ(\mu_{p^\infty})$), in terms of the corresponding special $L$-value $L(\Sym^2 F, k)$. However, this method depends crucially on the self-duality of the corresponding twist of the Galois representation; hence it cannot be used to study the Dirichlet-character twists $L(\Sym^2 F \otimes \psi, s)$ for nontrivial Dirichlet characters $\psi$, or to prove either inclusion in the Main Conjecture.
   
  \subsection{Our results}
  
   In this paper, we shall construct an Euler system for symmetric square Galois representations of modular forms. Let $\psi$ be a Dirichlet character. (We shall state our main theorems here assuming $\psi$ is $\Zp$-valued, in addition to the simplifying assumptions imposed on $F$ in the previous section; for the full statements see the main body of the paper.) Define
   \[ H^1_{\mathrm{Iw}}\left(\QQ(\mu_{mp^\infty}), \Sym^2 M_{\Zp}(F)^*(1)(\psi) \right)=\varprojlim_r H^1\left(\QQ(\mu_{mp^r}),\Sym^2 M_{\Zp}(F)^*(1)(\psi) \right),\]
   where the inverse limit is taken with respect to the Galois corestriction maps. 
   
   \begin{lettertheorem}
    [{Theorem \ref{thm:ES}}]
    For $F$ and $\psi$ as above, and any choice of auxilliary integer $c > 1$ coprime to $6 N_\psi$, there is a collection of cohomology classes
    \[ \mathrm{c}_m \in H^1_{\mathrm{Iw}}\left(\QQ(\mu_{mp^\infty}), \Sym^2 M_{\Zp}(F)^*(1)(\psi) \right) \]
    for all integers $m$ coprime to $c N_\psi$, which satisfy the norm-compatibility relation
    \[ \operatorname{cores}_m^{\ell m}(c_{\ell m}) = P_\ell(\psi(\ell) \Frob_\ell^{-1}) c_m \]
    for primes $\ell \nmid m p c N_\psi$. Here $P_\ell$ is the degree 4 polynomial
    \[ P_\ell(X) = (1 - \ell^{k-1} X)^2 \left(1 - (a_\ell(F)^2 - 2\ell^{k-1})X + \ell^{2k-2} X^2\right),\]
    and $\Frob_\ell$ is the Frobenius at $\ell$. Moreover, all the classes $c_m$ lie in the eigenspace where complex conjugation acts via $\psi(-1)$.
   \end{lettertheorem}
    
   Note that the polynomial $P_\ell$ is not the one predicted from the theory of Euler systems: it is the Euler factor of the four-dimensional tensor product $M_{\Zp}(F) \otimes M_{\Zp}(F)$ at $\ell$, not of its three-dimensional symmetric square summand. This is essentially because the $c_m$ arise as the images of cohomology classes in $M_{\Zp}(F)^* \otimes M_{\Zp}(F)^*$, the \emph{Beilinson--Flach elements} of \cite{KLZ1b}. 
   
   Theorem A is clearly vacuous as stated, since the $c_m$ could all be zero, but their nontriviality is guaranteed by the next theorem. Let $\nu = \psi(-1) \in \{ \pm 1\}$, and let $e_\nu$ denote the idempotent in $\Zp[[\Zp^\times]]$ projecting to the $\nu$-eigenspace for the action of the subgroup $\{\pm 1\}$ of $\Zp^\times$.
   
   \begin{lettertheorem}[{Theorem \ref{thm:explicitrecip}}]
    The localisation of $c_1$ at $p$ is non-trivial, and it is sent by Perrin-Riou's $p$-adic regulator map to
    \[
     \tag{\dag}
     e_\nu \left( \psi(-1) (c^2 - c^{2s-2k + 2} \psi(c)^{-2}) \cdot L_p(\Sym^2 F \otimes \psi, s) \cdot L_p(\psi, s - k + 1)\right) \]
    where $L_p(\psi, s)$ is the Kubota--Leopoldt $p$-adic $L$-function of $\psi$.
   \end{lettertheorem}
   
   This theorem depends crucially on Samit Dasgupta's factorisation formula for $p$-adic Rankin $L$-functions \cite{Dasgupta-factorization}. Interestingly, Dasgupta's proof of the factorisation formula also uses Beilinson--Flach elements, but relies on considering their projection to the \emph{alternating} square of $M_{\Zp}(F)$; the two projections are related by the functional equation of the symmetric square $L$-series.
   
   Our final result is an application of this Euler system to the Main Conjecture. This is greatly complicated by the fact that our classes have the ``wrong'' Euler factor appearing in their norm relations. Nonetheless, under some (rather restrictive) hypotheses, we can adapt the Kolyvagin--Rubin Euler system machine to give an upper bound for the size of the Selmer groups of twists of the symmetric square representation. In the present situation this can be stated as follows:
   
   \begin{lettertheorem}
    [{Theorem \ref{mainthm}}]
    Assume that:
    \begin{itemize}
     \item $p \ge 7$;
     \item the coefficients of $F$ and the values of $\psi$ lie in $\Qp$; 
     \item the image of the Galois representation 
     \[ \Gal(\QQbar / \QQ) \to \Aut M_{\Zp}(F) \cong \GL_2(\Zp) \]
     contains $\SL_2(\Zp)$;
     \item $\psi$ is not trivial or quadratic;
     \item $\psi(p) \ne 1$.
    \end{itemize}
    Then the characteristic ideal of $e_\nu \Sel(\QQ(\mu_{p^\infty}), \Sym^2 F \otimes \psi)$ divides $d \cdot A$, where $A$ is the product {\upshape{(\dag)}} and $d \in \Zp$ is any generator of the congruence ideal of $F$.
   \end{lettertheorem}
   
   This theorem is, therefore, an approximation to the ``Euler system divisibility'' in the main conjecture \eqref{eq:mainconj}, after projecting to the $\nu$-eigenspace. (This sign condition is not too onerous, since in fact the $+1$ and $-1$ eigenspaces are interchanged by the functional equation, so it is sufficient to prove the main conjecture for a single eigenspace.)
   
   Some of the above hypotheses are imposed largely for convenience and could in principle be relaxed substantially, but it is absolutely essential for our methods that $\psi$ be non-trivial, since otherwise we cannot get rid of the extra degree 1 factor in the polynomials $P_\ell$. This means that our results do not overlap with those obtained from the Taylor--Wiles method (which \emph{only} works when $\psi$ is trivial).

  \subsection{Organisation of the paper}
   
   We begin by recalling, in \S \ref{sect:Lfunc}, some known results concerning complex and $p$-adic symmetric square $L$-functions. Much of this material is very classical; the exception is Dasgupta's factorisation formula, a much more recent result, which we recall in \S \ref{sect:samit}. 
   
   In \S \ref{sect:selmer}, we define the algebraic counterparts of these objects -- Selmer groups with appropriate local conditions -- and formulate an Iwasawa main conjecture. This is fairly standard material for the experts, but we have not found a reference for these constructions in the present degree of generality, and we hope that our presentation of this material may serve as a useful introduction for the less expert reader.
   
   With these preliminaries out of the way, \S \ref{sect:ES} gives the construction of the Euler system, using the cohomology classes on products of modular curves constructed in our earlier works with Lei and Kings \cite{leiloefflerzerbes14, KLZ1a, KLZ1b}. In \S \ref{sect:kolyvagin} we apply this to bounding Selmer groups, using a small refinement of the Kolyvagin--Rubin machine which is explained in Appendix \ref{appendix}; we begin by bounding a Selmer group over the cyclotomic extension $\Qi$, and then obtain results for Selmer groups over $\QQ$ by descent.
   
   In \S \ref{sect:example} we conclude the paper with a numerical example involving $L(\Sym^2 F \otimes \psi, 22)$, where $F$ is the eigenform of weight 16 and level 1, and $\psi$ is a nontrivial even character of conductor 7.
   
  \subsection*{Acknowledgements}
   
   It is a pleasure to thank John Coates and Andrew Wiles for their interest in our work and helpful discussions; and Samit Dasgupta for explaining many aspects of \cite{Dasgupta-factorization} to us. We would also like to thank Henri Darmon for valuable comments on an earlier draft of this paper, and the anomymous referee for his suggestions on how to improve the exposition of the results.
   
 \section{L-functions}
  \label{sect:Lfunc}
  
  We begin by recalling some known results concerning classical and $p$-adic symmetric square $L$-functions.
  
  Let $f$ be a cuspidal, normalised, new modular eigenform, of level $N_f$, weight\footnote{With some trepidation, we have changed the normalisations from our previous work; $f$ now has weight $k$ rather than weight $k + 2$. This is for compatibility with \cite{Dasgupta-factorization}.} $k \ge 2$ and nebentypus $\varepsilon$, with coefficients in a number field $L \subset \CC$. We shall suppose throughout this paper that $f$ is not of CM type.
  
  \subsection{Complex L-functions}
   
   We introduce the various $L$-functions which we shall consider, following \cite{schmidt88, Dasgupta-factorization}. For a prime $\ell \nmid N$, we write 
   \[ 
    P_\ell(\Sym^2 f, X) = (1 - \alpha_\ell^2 X)(1 - \alpha_\ell \beta_\ell X)(1 - \beta_\ell^2 X) \in \cO_L[X],
   \]
   the local Euler factor of $\Sym^2 f$ at $\ell$, where (as usual) $\alpha_\ell$ and $\beta_\ell$ denote the roots of the Hecke polynomial $X^2 - a_\ell(f) X + \ell^{k - 1} \varepsilon(\ell)$. More generally, if $\chi$ is a Dirichlet character of conductor $N_\chi$ coprime to $\ell$, we write
   \[ P_\ell(\Sym^2 f \otimes \chi, X) = P_\ell(\Sym^2 f, \chi(\ell) X). \]
   
   \begin{definition}
    Let $L(\Sym^2 f \otimes \chi, s)$ be the \emph{primitive} symmetric square $L$-function, which is given by an Euler product (convergent for $\Re(s) > k$),
    \[ 
     L(\Sym^2 f \otimes \chi, s) = 
     \prod_{\ell} P_\ell(\Sym^2 f \otimes \chi, \ell^{-s})^{-1},
    \]
    where the local Euler factors at the ``bad'' primes $\ell \mid N_f N_\chi$ are as given in \cite{schmidt88}.
   \end{definition}

    Note that we do not include Euler factors at the infinite place.
   
   \begin{remark}
    This Dirichlet series is the $L$-series of the automorphic representation of $\GL_3(\AA_\QQ)$ associated to $\Sym^2 f \otimes \chi$ by Gelbart and Jacquet \cite{gelbartjacquet78}. Since $f$ is not of CM type, this automorphic representation is cuspidal, and the $L$-series therefore has analytic continuation to all $s \in \CC$. 
   \end{remark} 
   
   \begin{definition}
    We define the \emph{imprimitive} symmetric square $L$-function by
    \[ 
     L^\imp(\Sym^2 f, \chi, s) \coloneqq \prod_{\ell} \left[ (1 - \chi(\ell) \alpha_\ell^2 \ell^{-s})  (1 - \chi(\ell) \alpha_\ell \beta_\ell\ell^{-s}) (1 - \chi(\ell) \beta_\ell^2 \ell^{-s})\right]^{-1}
    \]
    where we understand one or both of $\{\alpha_\ell, \beta_\ell\}$ to be zero if $\ell \mid N_f$, and $\chi(\ell) = 0$ if $\ell \mid N_\chi$, so $L^\imp(\Sym^2 f, \chi, s)$ differs from $L(\Sym^2 f \otimes \chi, s)$ by finitely many Euler factors at the primes dividing $N_f N_\chi$. 
   \end{definition} 
    
    The importance of the imprimitive $L$-series is that it is given by a simple formula involving the coefficients of $f$:
     
   \begin{proposition}
    We have
    \begin{align*}
     L^\imp(\Sym^2 f, \chi, s) &= L_{N_f N_\chi}(\chi^2 \varepsilon^2, 2s - 2k + 2) \sum_{n \ge 1} a_{n^2}(f)\chi(n)  n^{-s}\\
     &= \frac{L_{N_f N_\chi}(\chi^2 \varepsilon^2, 2s - 2k + 2)}{L_{N_f N_\chi}(\chi \varepsilon, s - k + 1)} \sum_{n \ge 1} a_n(f)^2 \chi(n) n^{-s}.
    \end{align*}
    Here $L_{N_f N_\chi}(-)$ denotes the Dirichlet $L$-series with the factors at primes dividing $N_f N_\chi$ omitted.
   \end{proposition}
   
   In some simple cases, we in fact have $L^\imp = L$ (this occurs, for instance, if $f$ has trivial character and square-free level, and $(N_f, N_\chi) = 1$). In general, the relation between primitive and imprimitive $L$-series is as follows:
   
   \begin{proposition}[Schmidt]
    \label{prop:imp-zeros}
    If $f$ has minimal level among its Dirichlet-character twists, then the ratio
    \[ 
     \frac{L^\imp(\Sym^2 f, \chi, s)}{L(\Sym^2 f \otimes \chi, s)}
    \]
    is an entire function (a product of polynomials in the variables $\ell^{-s}$ for $\ell \mid N_f N_\chi$) whose only possible zeroes lie on the line $\Re(s) = k-1$.
   \end{proposition}
   
   \begin{remark}
    \label{rmk:nonvanishing}
    Note that $L(\Sym^2 f \otimes \chi, s)$ is non-zero for integers $s \ge k$. For $s > k$ this is immediate from the absolute convergence of the Euler product. For $s = k$ it is a consequence of a much deeper non-vanishing theorem for $L$-functions of cuspidal automorphic representations of $\GL_n$ due to Jacquet and Shalika, cf.~\cite{jacquetshalika76}. If $f$ has minimal level among its twists, then the corresponding non-vanishing result for the imprimitive $L$-function follows from this, by Proposition \ref{prop:imp-zeros}.
   \end{remark}
    
   We also consider the Rankin--Selberg convolution $L$-function of $f$ with itself, $L(f \otimes f \otimes \chi, s)$, and its imprimitive analogue $L^\imp(f, f, \chi, s)$. These are given respectively by
   \[ L(f \otimes f \otimes \chi, s) \coloneqq L(\Sym^2 f \otimes \chi, s) \cdot L(\chi \varepsilon, s - k + 1) \]
   and
   \[ L^\imp(f, f, \chi, s) \coloneqq L^\imp(\Sym^2 f \otimes \chi, s) \cdot L_{N_f N_\chi}(\chi \varepsilon, s - k + 1). \]
   Thus we have
   \[  
    L^\imp(f, f, \chi, s) = L_{N_f N_\chi}(\chi^2 \varepsilon^2, 2s - 2k + 2) \sum_{n \ge 1} a_n(f)^2 \chi(n) n^{-s}.
   \]

  
  \subsection{Rationality results}
   
   We now recall a rationality result for the critical values of the symmetric square $L$-function due to Schmidt \cite{schmidt88}, building on earlier work of Sturm \cite{sturm80} and Coates--Schmidt \cite{coatesschmidt87}. The following computation is standard:
      
   \begin{proposition}
    The value $L(\Sym^2 f \otimes \chi, s)$ is critical if $s$ is an integer in the range $\{1, \dots, k-1\}$ with $(-1)^s = -\chi(-1)$, or if $s$ is an integer in the range $\{k, \dots, 2k-2\}$ with $(-1)^s = \chi(-1)$.
   \end{proposition}
   
   We now recall a rationality result for these $L$-values.
   
   \begin{notation}
    \begin{enumerate}[(i)]
     \item Let $\langle f, f \rangle$ denote the Petersson norm of $f$,
     \[ 
      \langle f, f \rangle \coloneqq \int_{\Gamma_1(N_f) \backslash \cH} |f(x + iy)|^2 y^{k-2} \, \mathrm{d}x\, \mathrm{d}y,
     \]
     where $\cH$ is the upper half-plane. 
     
     \item For $\chi$ a Dirichlet character of conductor $N_\chi$, let $G(\chi)$ be the Gauss sum defined by
     \[ 
      G(\chi) \coloneqq \sum_{a \in (\ZZ / N_\chi \ZZ)^\times} \chi(a) \exp\left(\frac{2 \pi i a}{N_\chi}\right).
     \]
     (We adopt the convention that $G(\chi) = 1$ if $\chi$ is the trivial character.) 
    
     \item For $\sigma \in G_{\QQ}$, let $f^\sigma$ be the newform obtained by applying $\sigma$ to the $q$-expansion coefficients of $f$.
    \end{enumerate}
   \end{notation}

   \begin{theorem}[Schmidt, \cite{schmidt88}]
    Let $s$ be a critical value of $L(\Sym^2 f \otimes \chi)$. Define
    \[ 
     I(f, \chi, s) \coloneqq \frac{L(\Sym^2 f \otimes \chi, s)}{\pi^{k-1} \langle f, f \rangle} \left( \frac{G(\chi^{-1} \varepsilon^{-1})}{(2\pi i)^{s - k + 1}}\right)^{1 + \delta},
    \]
    where $\delta = 0$ if $1 \le s \le k-1$, and $\delta = 1$ if $k \le s \le 2k-2$. Then $I(f, \chi, s) \in \QQbar$, and $I(f, \chi, s)^\sigma = I(f^\sigma, \chi^\sigma, s)$ for any $\sigma \in G_{\QQ}$. 
   \end{theorem}
   
   \begin{remark}
    Note that our $I(f, \chi, s)$ is $I(s - k + 1, f, \chi^{-1} \varepsilon^{-1})$ in Schmidt's notation.
   \end{remark}
   
   We clearly have an identical result with $L(\Sym^2 f \otimes \chi, s)$ replaced by its imprimitive analogue, since the ratio of the two $L$-functions is a rational function of the coefficients of $f$ and the values of $\chi$, and thus depends Galois-equivariantly on $f$ and $\chi$.
   
   \begin{remark}
    Let $\Omega^+_f$ and $\Omega^-_f$ denote the periods of $f$, normalised as in \cite[\S 16.1]{kato04}; thus $\Omega^+$ and $\Omega^-$ are unique up to multiplication by $L^\times$, and if $f$ has real coefficients, then $\Omega^+_f \in \RR$ and $\Omega^-_f \in i\RR$. Then we have the relation
    \[\Omega^+ \Omega^- =  
      i^{k-1} \pi^{2k-2} G(\varepsilon) \langle f, f \rangle   \bmod L^\times.
    \]
    In particular, when $f$ corresponds to an elliptic curve $E$, we may take $\Omega_f^+$ and $\Omega_f^-$ to be the real and imaginary Ner\'on periods of $E$, and we recover the rationality statements formulated in \cite{dabrowskidelbourgo97}.
   \end{remark}


  \subsection{P-adic L-functions}
  
   Let $p \ge 5$ be a prime with $p \nmid N_f$, and $\frP$ a prime of the field $L$ above $p$. We assume that $f$ is ordinary at $\frP$ (i.e.~that $v_{\frP}(a_p(f)) = 0$). Denote by $\alpha_p$ the unique root of the Hecke polynomial at $p$ that lies in $\cO_{L, \frP}^\times$.
   
   Let $\psi$ be a Dirichlet character of conductor coprime to $p$; enlarging $L$ if necessary, we assume $\psi$ takes values in $L^\times$.
   
   \begin{notation}
    Let $\Gamma = \Zp^\times$, and let $\Lambda(\Gamma) = \Zp[[\Gamma]]$ be the Iwasawa algebra of $\Gamma$. We identify $\Lambda(\Gamma)$ with the ring of $\Zp$-valued analytic functions on the rigid space $\cW$ parametrising continuous characters of $\Gamma$. (This space $\cW$ is often called ``weight space''.)
    
    We shall write characters of $\Gamma$ additively, and identify an integer $s$ with the character $x \mapsto x^s$, so ``$s + \chi$'' denotes the character $x \mapsto x^s \chi(x)$.
   \end{notation}

   \begin{theorem}[Schmidt, Hida, Dabrowski--Delbourgo]
    \label{thm:padicL}
    There exists an element
    \[ L_p(\Sym^2 f \otimes \psi) \in \Lambda(\Gamma) \otimes_{\Zp} \CC_p, \]
    the $p$-adic $L$-function of $\Sym^2f \otimes \psi$, with the following interpolation property.
    \begin{enumerate}
     \item[(i)] If $1 \le s \le k-1$ and $\chi$ is a finite-order character of $\Gamma$ such that $(-1)^s\chi(-1) = -\psi(-1)$, then 
     \[ 
      L_p(\Sym^2 f \otimes \psi)(s + \chi) 
      = \frac{(-1)^{s - k + 1} \chi(-1) \Gamma(s)}{2^{2k} i^a}
      \mathcal{E}_p(s, \chi) \frac{G(\chi)}{(2 \pi i)^{s - k + 1}}
      \frac{L(\Sym^2 f \otimes \psi \chi^{-1}, s)}{\pi^{k - 1} \langle f, f \rangle},
     \]
     where $a = 0$ if $\psi(-1) = (-1)^k$ and $a = 1$ if $\psi(-1) = (-1)^{k + 1}$, and the factor $\mathcal{E}_p(s, \chi)$ is defined by
     \[ 
      \begin{cases}
       \left(p^{s-1} \psi(p)^{-1} \alpha_p^{-2}\right)^r 
       &  \text{if $\chi$ has conductor $p^r > 1$,}\\
       (1 - p^{s-1} \psi(p)^{-1}\alpha_p^{-2})
       (1 - \psi(p) \alpha_p \beta_p  p^{-s})
       (1 - \psi(p) \beta_p^2 p^{-s}) 
       & \text{if $\chi = 1$.}
      \end{cases}
     \]
     \item[(ii)] If $k \le s \le 2k-2$ and $\chi$ is a finite-order character of $\Gamma$ such that $(-1)^s\chi(-1) = \psi(-1)$, then
     \[ 
      L_p(\Sym^2 f \otimes \psi)(s + \chi) = 
      \frac{\Gamma(s - k + 1)\Gamma(s)}{2^{2s + 1}} \mathcal{E}_p'(s, \chi) \chi(N_{\psi \varepsilon})^2 G(\chi)^2 \frac{L(\Sym^2 f \otimes \psi \chi^{-1}, s)}{\pi^{2s - k + 1} \langle f, f \rangle}      
     \]
     where the Euler factor $\mathcal{E}_p'(s, \chi)$ is defined by
     \[ 
      \begin{cases}
       \left(p^{s-1} \psi(p)^{-2} \varepsilon(p)^{-1} \alpha_p^{-2}\right)^r 
       &  \text{if $\chi$ has conductor $p^r > 1$,}\\
       (1 - p^{s-1} \psi(p)^{-1}\alpha_p^{-2})
       (1 - p^{s-1} \psi(p)^{-1} \alpha_p^{-1} \beta_p^{-1})
       (1 - \psi(p) \beta_p^2 p^{-s}) 
       & \text{if $\chi = 1$.}
      \end{cases}
     \]  
    \end{enumerate}
   \end{theorem}
   
   \begin{remark}\mbox{~}
    \begin{enumerate}[(i)]
     
     \item Our formulation of the interpolation property follows \cite[Theorem 3.4]{Dasgupta-factorization}. Dasgupta does not state an interpolation formula in the case when $s \ge k$ and $\chi$ is trivial, which is actually the most interesting case for us; but such a formula can be extracted from the case $s \le k-1$ together with the functional equations of the complex and $p$-adic $L$-series, cf.~\cite[Theorems 10.1 and 10.2]{Dasgupta-factorization}.
     
     \item The $p$-adic $L$-function can be descended to a much smaller field of coefficients than $\CC_p$. In fact it is immediate from Schmidt's construction that $L_p(\Sym^2 f \otimes \psi)$ lies in $\Lambda(\Gamma) \otimes L_{\frP}(\zeta_N)$ where $N$ is the conductor of $\varepsilon\psi$; and automorphisms of $L_{\frP}(\zeta_N) / L_{\frP}$ act on $L_p(\Sym^2 f \otimes \psi)$ as multiplication by values of the character $\varepsilon\psi$, which are $p$-adic units, so  $L_p(\Sym^2 f \otimes \psi)$ defines a fractional ideal in the ring $\cO_{L, \frP}[[\Gamma]]$.
     
     \item The history of this theorem is somewhat convoluted. In \cite{schmidt88}, Schmidt initially defined a $p$-adic measure interpolating values of the \emph{imprimitive} $L$-function, and then defined a primitive $p$-adic $L$-function by dividing out by the imprimitive factor, giving an element which \emph{a priori} lies in the field of fractions of $\Lambda(\Gamma)$. Hida showed in \cite{hida90} that Schmidt's $L$-function had no poles, but was unable to show that the interpolating property was valid if $s + \chi$ was a zero of the imprimitive factor. Dabrowski and Delbourgo in \cite{dabrowskidelbourgo97} claimed a proof of the holomorphy and interpolating property in full generality (as well as allowing some non-ordinary $f$), but their proof appears to be incomplete; however, the cases where the arguments of Dabrowski--Delbourgo are unclear are confined to non-algebraic characters, so combining the arguments of these three papers does appear to prove the desired theorem.
    \end{enumerate}
   \end{remark}
   
   \begin{proposition}
    The $p$-adic multipliers $\mathcal{E}_p(s, \chi)$ and $\mathcal{E}_p'(s, \chi)$ vanish if and only if $\chi = 1$, $\varepsilon\psi(p) = 1$, and $s = k-1$ (for $\mathcal{E}_p$) or $s = k$ (for $\mathcal{E}_p'$). In particular, if $s, \chi$ are as in (ii), then $L_p(\Sym^2 f \otimes \psi)(s + \chi)$ is non-zero, except in the exceptional zero case where $\mathcal{E}_p'$ vanishes.
   \end{proposition}
   
   \begin{proof}
    The vanishing criterion for $\mathcal{E}_p(s, \chi)$ and $\mathcal{E}_p'(s, \chi)$ is an easy case-by-case check. The non-vanishing result for the $p$-adic $L$-function in the range $s \ge k$ follows from this together with Remark \ref{rmk:nonvanishing}.
   \end{proof}
      
   We will also need the following remark:
   
   \begin{proposition}[Hida]
    Let $I(f) \trianglelefteq \cO_{L, \frP}$ be the congruence ideal of $f$ at $\frP$. Then for any $d \in I$, we have 
    \[ d \cdot L_p(\Sym^2 f \otimes \psi) \in \Lambda(\Gamma) \otimes_{\Zp} \cO_{\CC_p}.\]
   \end{proposition}
%
%
%
%
%
   
         
  \subsection{Dasgupta's factorization formula}
   \label{sect:samit}
  
   We now recall the main theorem of \cite{Dasgupta-factorization}:
   
   \begin{theorem}[Dasgupta]
    We have the following identity of analytic functions on $\cW$:
    \[ 
     L_p(f \otimes f \otimes \psi, s) = 
     L_p(\Sym^2 f \otimes \psi, s) L_p(\psi \varepsilon, s - k + 1). 
    \]
   \end{theorem}
   
   Since it is the imprimitive $L$-function which interests us here, we shall chiefly be interested in the following consequence. We let $L_p^{\imp}(\Sym^2 f \otimes \psi) \in \Lambda(\Gamma) \otimes \CC_p$ be the unique analytic function interpolating the values of the imprimitive symmetric square $L$-series, and similarly for the $p$-adic Rankin $L$-series. Then we have
   \[ 
    L_p^\imp(f \otimes f \otimes \psi, s) = 
    L_p^\imp(\Sym^2 f \otimes \psi, s) L_{p, N_f N_\psi}(\psi \varepsilon, s - k + 1).
   \]
  
 
 \section{Selmer groups}
  \label{sect:selmer}
 
  In this section we define various Selmer groups attached to the symmetric square Galois representation, and formulate an Iwasawa main conjecture relating these to $p$-adic $L$-functions. The methods we use here are well-known to the experts, but we have chosen to spell out the constructions in detail, in the hope that this will be a useful guide to the less expert reader.
 
  \subsection{The Galois representation of \texorpdfstring{$f$}{f}}
   \label{sect:galrep}
    
   Let $p$ be an odd prime, and let $\frP$ be a prime of $L$ above $p$. We denote by $M_{L_\frP}(f)$ the Galois representation associated to $f$, which is canonically realised as a direct summand of the \'etale cohomology group 
   \[ H^1_{\et, c}(Y_1(N_f)_{\QQbar}, \Qp) \otimes_{\Qp} L_{\frP}. \]
   
   \begin{remark}
    As an abstract Galois representation, $M_{L_\frP}(f)$ is characterised by the condition that
    \[ \sideset{}{_{L_\frP}}\det\left(1 - X \Frob_\ell^{-1} : M_{L_\frP}(f)\right) = (1 - \alpha_\ell X)(1 - \beta_\ell X) \]
    for primes $\ell \nmid N_f p$, where $\Frob_\ell$ is the arithmetic Frobenius. 
   \end{remark} 
    
    There is a canonical $G_{\QQ}$-stable $\cO_{L, \frP}$-lattice $M_{\cO_{L, \frP}}(f)^*$ in $M_{L_\frP}(f)^*$, generated by the image of the \'etale cohomology of $Y_1(N_f)_{\QQbar}$ with $\Zp$-coefficients.
    
   \begin{remark}
    As in \cite{leiloefflerzerbes14, KLZ1a, KLZ1b}, we consider $Y_1(N)$ as a moduli space for elliptic curves with a point of order $N$, \emph{not} an embedding of $\mu_N$, so the cusp at $\infty$ is not defined over $\QQ((q))$ with our normalisations. Thus the action of $G_{\QQ}$ on modular forms determined by identifying them with classes in the de Rham cohomology of $Y_1(N)$ does not coincide with the action on $q$-expansions.
   \end{remark}
    
   We adopt the following convention regarding Dirichlet characters: if $\chi: (\ZZ / m \ZZ)^\times \to L^\times$ is a Dirichlet character, then we interpret $\chi$ as a character of $G_{\QQ}$ by composing $\chi$ with the mod $m$ cyclotomic character, so we have $\chi(\Frob_\ell) = \chi(\ell)$ for $\ell \nmid m$. With this convention, we have $\det M_{L_\frP}(f) = L_{\frP}(1-k)(\varepsilon^{-1})$.
   
   
  \subsection{The symmetric square representation}
  
   Let $\psi$ be a Dirichlet character of conductor $N_\psi$ coprime to $p$. Enlarging the field $L$ if necessary, assume that $\psi$ takes values in $L^\times$. 
   
   \begin{definition}
     Define
    \[ V \coloneqq  \Sym^2 M_{L_\frP}(f)^*(1 + \psi), \]
    where we identify $\psi$ with a character of $G_{\QQ}$ as above.
    Let $T$ be the lattice in $V$ corresponding to $M_{\cO_{L, \frP}}(f)^*$. 
   \end{definition}
   
   \begin{note} 
    The space $V$ is a 3-dimensional irreducible $L_{\frP}$-linear rep\-re\-sen\-tation of $G_{\QQ}$, unramified outside $p N_f N_\psi$ and crystalline at $p$. 
   \end{note} 
       
   Recall from \cite[\S 7.2]{KLZ1b} that the representation $M_{L_\frP}(f)^*$ has a 2-step filtration by $L_{\frP}$-subspaces stable under $G_{\Qp}$; the graded pieces have Galois actions $\operatorname{unram}(\alpha)$ and $\kappa^{k-1}\times \operatorname{unram}\left( \tfrac{\beta}{p^{k-1}} \right)$, respectively.     Here $\kappa$ denotes the cyclotomic character, and $\operatorname{unram}(x)$ denotes the unramified character mapping arithmetic Frobenius to $x$. 
   
   Thus $V = \Sym^2 M_{L_{\frP}}(f)^*(1 + \psi)$ has a 3-step filtration: we have
   \begin{equation}
    \label{eq:filtration} 
    V = \sF^0 V \supset \sF^1 V \supset \sF^2 V \supset \sF^3 V = 0,
   \end{equation} 
   with all graded pieces 1-dimensional. 
   
   \begin{lemma}
    The Galois actions on the graded pieces $\Gr^i V \coloneqq \sF^i V / \sF^{i + 1} V$ are as follows:
    \begin{align*}
     \Gr^0 V &= \kappa \times \operatorname{unram}(\alpha^2 \psi(p)), \\
     \Gr^1 V &= \kappa^k \times \operatorname{unram}( \varepsilon(p) \psi(p)),\\
     \Gr^2 V &= \kappa^{2k-1} \times 
     \operatorname{unram}\left( \tfrac{\beta^2}{p^{2k-2}} \psi(p) \right).
    \end{align*}
   \end{lemma} 
   \qed
   
   We use the same notations for the induced filtration on $T$ by $\cO_{L, \frP}$-submodules. 
   
   \begin{notation}
    Write
    \begin{align*}
     A \coloneqq& \Sym^2 M_{\cO_{L, \frP}}(f)(\psi^{-1}) \otimes \Qp/\Zp\\
     =& T^\vee(1),
    \end{align*}
    where $(-)^\vee$ denotes Pontryagin dual. 
   \end{notation} 
    
    We equip $A$ with a 3-step filtration by defining $\sF^i A$ to be the orthogonal complement of $\sF^{3-i} T$.
  
     
  \subsection{Local cohomology groups}

   We now recall the well-known Bloch--Kato subspaces of local Galois cohomology, introduced in \cite{blochkato90}:
   
   \begin{definition}
    \label{def:h1f}
    For $M$ any $\Qp$-linear representation of $G_{K}$, where $K / \QQ_\ell$ is a finite extension, we define
    \[
     \Hf(K, M) = 
     \begin{cases}
      \ker H^1(K, M) \to H^1(K, M \otimes \BB_{\cris}) & \text{if $\ell = p$},\\
      \ker H^1(K, M) \to H^1(K^{\mathrm{nr}}, M) & \text{if $\ell \ne p$}.
     \end{cases}
    \]
    If $M_0$ is a $G_K$-stable $\Zp$-lattice in $M$, we let $\Hf(K, M_0)$ (respectively $\Hf(K, M/M_0)$) denote the preimage (resp.~image) of $\Hf(K, M)$.
   \end{definition}
   
   We shall now show that for our representation $V$, the $\Hf$ local condition at $p$ can be computed in terms of the filtration $\sF^\bullet$. First we need a lemma:
   
   \begin{notation}
    Let $\QQ_{p, r} = \Qp(\mu_{p^r})$, and $\Qpi = \bigcup_{r \ge 1} \QQ_{p, r}$.
   \end{notation}

   \begin{lemma}
    \label{lemma:cyclo-invariants}
    We have $H^0(\Qpi, \Gr^0 V) = H^0(\Qpi, \Gr^2 V) = 0$, while $H^0(\Qpi, \Gr^1 V)$ is zero if and only if $\varepsilon(p) \psi(p) \ne 1$.
   \end{lemma}
   
   \begin{proof} 
    We have described each of the graded pieces in the form $\kappa^r \operatorname{unram}(\lambda)$, for some $r \in \ZZ$ and $\lambda \in \cO_{L, \frP}^\times$. Such a character is trivial over $\Qpi$ if and only if $\lambda = 1$. Since $\alpha$ and $\beta$ have complex absolute value $p^{(k-1)/2}$, and $\psi(p)$ is a root of unity, neither $\alpha^2\psi(p)$ nor $p^{2-2k} \beta^2\psi(p) $ can be equal to 1.
   \end{proof}   
   
   Then we have the following computation, which is immediate from the standard dimension formulae for $\Hf$ (cf.~\cite[Lemma 4.1.7]{fukayakato06}). We let $j$ be an integer, and $\chi$ a finite-order character of $\Gamma$.
   
   \begin{proposition}
    \label{prop:BKSel}
    Let $m$ be the following integer:
    \[
     m = 
     \begin{cases}
      0 & \text{if $j \le 0$}\\
      1 & \text{if $1 \le j \le k-1$}\\
      2 & \text{if $k \le j \le 2k-2$}\\
      3 & \text{if $j \ge 2k-1$.}
     \end{cases}
    \]
    If $j = k$ or $j = k-1$, suppose that either $\chi \ne 1$ or $\varepsilon(p) \psi(p) \ne 1$. Then we have the following statements:
    \begin{enumerate}[(i)]
     \item The natural map
     \[ H^1(\Qp, \sF^m V(-j-\chi)) \rTo H^1(\Qp, V(-j-\chi)) \]
     is injective.
     \item The image of this map is the subspace $\Hf(\Qp, V(-j-\chi))$.
    \end{enumerate}
   \end{proposition}
   
   \begin{remark}
    One can check that in the exceptional case $\chi = 1$ and $\varepsilon(p)\psi(p) = 1$, the image of the map of (i) is the Bloch--Kato $H^1_{\mathrm{e}}$ subspace if $j = k$, and $H^1_{\mathrm{g}}$ if $j = k-1$, but we shall not use this fact here.
   \end{remark}
   
   We have a related statement in the limit over the cyclotomic tower.
   
   \begin{definition}
    For $M$ a $\Qp$-linear representation of $G_{\Qp}$, and $M_0$ a $G_{\Qp}$-stable $\Zp$-lattice in $M$, we set
    \[ 
     H^1_{\Iw}(\Qpi, M) = \Qp \otimes_{\Zp} \varprojlim_r H^1(\QQ_{p, r}, M_0) 
    \]
    and
    \[ H^1_{\Iw, \mathrm{f}}(\Qpi, M) = \Qp \otimes_{\Zp} \varprojlim_r \Hf(\Qp, M_0).\]
   \end{definition}
   
   The groups $H^1_{\Iw}(\Qpi, V(-j-\chi))$ are independent of $j$ and $\chi$, in the sense that 
   \begin{equation}
    \label{eq:twist}
    H^1_{\Iw}(\Qpi, V(-j-\chi)) =  H^1_{\Iw}(\Qpi, V)(-j-\chi)
   \end{equation}
   as a $\Lambda(\Gamma)$-module. The groups $H^1_{\Iw, \mathrm{f}}(\Qpi,V(-j-\chi))$ are similarly independent of $\chi$, but they are \emph{not} independent of $j$.
   
   \begin{proposition}
    \label{prop:BKSelIw}
    Let $m$ be defined in terms of $j$ as in Proposition \ref{prop:BKSel}. Assume $\varepsilon(p) \psi(p) \ne 1$. Then $H^1_{\Iw, \mathrm{f}}(\Qpi, V(-j))$ is the image of the injective map
    \[ 
     H^1_{\Iw}(\Qpi, \sF^m V(-j)) \into  H^1_{\Iw}(\Qpi, V(-j)).
    \]
   \end{proposition}
   
   \begin{proof}
    This may appear to be an immediate consequence of Proposition \ref{prop:BKSel}, but it is not quite. The issue is that the natural map
    \[ H^1_{\Iw}(\Qpi, M) \to \varprojlim_r H^1(\QQ_{p, r}, M), \]
    for $M$ a $\Qp$-linear representation of $G_{\Qp}$, can fail to be injective. However, this map is injective when $H^0(\Qpi, M) \ne 0$, cf.~\cite[Propostion 8.1.5]{KLZ1b}. Applying this to $M = V(-j)/\sF^m$, which has zero invariants over $\Qpi$ by Lemma \ref{lemma:cyclo-invariants}, we see that any $x \in H^1_{\Iw, \mathrm{f}}(\Qpi, V(-j))$ must map to zero in $H^1_{\Iw}(\Qpi, V(-j)/\sF^m)$ as required.
   \end{proof}
    
    
  \subsection{Selmer groups}
   
   We recall the definition of the Bloch--Kato Selmer group for a global Galois representation:

   \begin{definition}
    If $K$ is a number field, and $M$ is a $\Zp[G_K]$ module isomorphic as an abelian group to $\Zp^n$, $\Qp^n$, or $(\Qp/\Zp)^n$ for some $n$, we define the \emph{Bloch--Kato Selmer group}
    \[ \Hf(K, M) \coloneqq \{ x \in H^1(K, M): \loc_v x \in \Hf(K_v, M) \text{ for all primes $v$}\}\]
    where the local subspaces $\Hf(K_v, M)$ are as in definition \ref{def:h1f}.
   \end{definition}
   
   In particular, for $j \in \ZZ$ and $\chi$ a finite-order character, we have the Selmer groups $\Hf(\QQ, A(j + \chi))$ and $H^1_{\mathrm{f}}(\QQ, T(-j -\chi))$. On the other hand, we also have Selmer groups defined using the filtrations $\sF^i$, an approach pioneered by Greenberg \cite{greenberg89}:

   \begin{definition}
    \label{def:GreenbergSel}
    For $0 \le i \le 3$, and $\tau$ any continuous character of $\Gamma$, we define the \emph{Greenberg Selmer group}
    \[
     H^1_{\Gr, i}(\QQ, A(\tau)) \coloneqq
     \left\{
      x \in H^1(\QQ, A(\tau))
      :
      \begin{array}{l}
       \loc_\ell x \in H^1_\mathrm{nr}(\QQ_\ell, A(\tau))\text{ for $\ell \ne p$},\\
       \loc_p x \in \operatorname{image} H^1(\Qp, \sF^i A(\tau))
      \end{array}
     \right\},
    \]
    and similarly $H^1_{\Gr, i}(\QQ, T(\tau^{-1}))$.
   \end{definition}   
   
   Then the following is immediate from Proposition \ref{prop:BKSel}:
   
   \begin{proposition}
    \label{prop:globalBKSel}
    Let $m$ be defined as in Proposition \ref{prop:BKSel}, and suppose that either $\varepsilon(p)\psi(p) \ne 1$ or $j + \chi \notin \{k-1, k\}$. Then we have inclusions
    \begin{align*}
     H^1_{\mathrm{Gr}, m}(\QQ, T(-j-\chi)) &\subseteq \Hf(\QQ, T(-j-\chi)), \\ 
     \Hf(\QQ, A(j + \chi)) &\subseteq H^1_{\mathrm{Gr}, 3-m}(\QQ, A(j + \chi));
    \end{align*}
    and both quotients are finite.\qed
   \end{proposition}

   We also have a version over $\Qi = \QQ(\mu_{p^\infty})$: we define
   \begin{align*}
    \Hf\left(\Qi, A(j + \chi) \right) 
    &\coloneqq \varinjlim_r \Hf\left(\QQ(\mu_{p^r}), A(j + \chi)\right),\\
    H^1_{\Iw, \mathrm{f}}\left(\Qi, T(-j-\chi) \right) 
    &\coloneqq \varprojlim_r \Hf\left(\QQ(\mu_{p^r}), T(-j-\chi)\right),
   \end{align*}
   and similarly for the Greenberg Selmer groups. It is immediate from the definitions that the groups over $\Qi$ are independent of $\chi$ (in the same sense as \eqref{eq:twist} above).
   
   \begin{proposition} \mbox{~}
    \begin{enumerate}[(i)]
     \item For any $i \in \{0, \dots, 3\}$ we have
     \[
      H^1_{\Iw, \Gr, i}(\Qi, T) = \left\{
      x \in H^1_{\Iw}(\Qi, T)
      :
      \loc_p x \in H^1_{\Iw}(\Qpi, \sF^i T)
     \right\}
     \]     
     and
     \[
      H^1_{\Gr, i}(\Qi, A) = \left\{
       x \in H^1(\Qi, A)
       :
       \begin{array}{l}
        \loc_\ell x = 0 \text{ for $\ell \ne p$},\\
        \loc_p x \in \operatorname{image} H^1(\Qpi, \sF^i A)
       \end{array}
      \right\}
     \]
     
     \item For any $i \in \{0, \dots, 3\}$, and $\tau$ any continuous character of $\Gamma$, we have isomorphisms of $\Zp[[\Gamma]]$-modules
     \[ H^1_{\Gr, i}(\Qi, A(\tau)) 
     \cong H^1_{\Gr, i}(\Qi, A)(\tau)\]
     and
     \[ H^1_{\Iw, \Gr, i}(\Qi, T(\tau)) 
          \cong H^1_{\Iw, \Gr, i}(\Qi, T)(\tau).
     \]
     
     \item If $\varepsilon(p) \psi(p) \ne 1$, then for $j \in \ZZ$ and $m$ as in Proposition \ref{prop:BKSel}, the maps
     \begin{align*} 
      H^1_{\Iw, \mathrm{Gr}, m}(\Qi, T(-j)) &\to H^1_{\Iw, \mathrm{f}}(\Qi, T(-j))\\
      \text{and } 
      \Hf(\Qi, A(j)) &\to H^1_{\Gr, 3-m}(\Qi, A(j))
     \end{align*}
     arising from Proposition \ref{prop:globalBKSel} are isomorphisms.
   \end{enumerate}
   \end{proposition}
   
   \begin{proof}
    To prove (i) for $T$, it suffices to note that any class in $H^1_{\Iw}(\Qi, M)$, where $M$ is a finitely-generated $\Zp$-module with $G_{\QQ}$-action, must automatically be unramified outside $p$; this is a well-known fact, see e.g.~\cite[\S B3.3]{rubin00}. The result for $A$ now follows by local Tate duality (or it can be seen directly from the fact that for $\ell \ne p$ the group $\Gal(\overline{\FF}_\ell / \FF_\ell(\mu_{p^\infty}))$ has trivial $p$-Sylow subgroup). From this description of the Greenberg cohomology groups, the twist-compatibility (ii) is clear.
    
    We now prove part (iii) for $T$. Using part (i) it suffices to check that we have
    \[ H^1_{\Iw}(\Qpi, \sF^m T(-j)) = H^1_{\Iw, \mathrm{f}}(\Qpi, T(-j))\]
    as submodules of $H^1_{\Iw}(\Qpi, T(-j))$. As Iwasawa cohomology groups are $p$-torsion-free, we may check this after inverting $p$, and this is exactly the statement of Proposition \ref{prop:BKSelIw}. The statement for $A$ follows from this together with local Tate duality.
   \end{proof}
    
  \subsection{The main conjecture}
  
   We now formulate a version of the Iwasawa main conjecture for our Selmer groups. For brevity we write $\cO = \cO_{L, \frP}$, and $\Lambda_\cO = \cO \otimes_{\Zp} \Zp[[\Gamma]]$.
    
   \begin{lemma}
    Let $\eta$ be a finite-order character of $\Gamma_{\mathrm{tors}}$, and $e_\eta$ the  corresponding idempotent in $\Lambda(\Gamma)$. Then, for any $i \in \{0, \dots, 3\}$,
    \begin{enumerate}[(i)]
     \item The group $e_\eta H^1_{\Iw, \Gr, i}(\Qi, T)$ is a torsion-free $e_\eta \Lambda_\cO(\Gamma)$-module of finite rank. If $e_\eta H^0(\Qi, T/pT)$ is zero, then it is free.
     
     \item We have the equality of ranks of $e_{\eta} \Lambda_{\cO}$-modules,
     \[ 
      \operatorname{rank} e_\eta H^1_{\Iw, \Gr, i}(\Qi, T) - \operatorname{rank} e_\eta H^1_{\Gr, 3-i}(\Qi, A)^\vee = 
      \begin{cases}
       2-i & \text{ if $\eta(-1) = \psi(-1)$,}\\
       1-i & \text{if $\eta(-1) = -\psi(-1)$.}
      \end{cases}
     \]
    \end{enumerate}
   \end{lemma}

   One expects that the ranks of the cohomology groups should be ``as small as possible'' given the constraint imposed by part (ii) of the lemma; in particular, at least one of the two terms on the left-hand side of (ii) should always be 0.
   
   We now impose, for simplicity, the assumption that $\varepsilon(p) \psi(p) \ne 1$.
   
   \begin{definition}
    We let $L^{\mathrm{alg}}(\Sym^2 f \otimes \psi)$ be the element of $\Lambda_{\cO}$, unique up to multiplication by a unit, whose image in $e_\eta \Lambda_\cO$ is the characteristic ideal of $e_\eta H^1_{\Gr, 1}(\Qi, A)^\vee$ if $\eta(-1) = \psi(-1)$, and of $e_\eta H^1_{\Gr, 2}(\Qi, A)^\vee$ if $\eta(-1) = -\psi(-1)$.
   \end{definition}
   
   We now make the following conjecture:
   
   \begin{conjecture}[Main Conjecture for $\Sym^2 f \otimes \psi$]
    We have
    \[ L_p^{\mathrm{alg}}(\Sym^2 f \otimes \psi) = d \cdot L_p(\Sym^2 f \otimes \psi)\]
    up to a unit in $\Lambda_{\cO}(\Gamma)$, where $d \in \cO$ is any generator of the congruence ideal of $f$.
   \end{conjecture}
   
   \begin{remark}
    A somewhat tortuous exercise involving Poitou--Tate duality shows that the algebraic $p$-adic $L$-function satisfies a functional equation relating $L_p^{\mathrm{alg}}(\Sym^2 f \otimes \psi, s)$ to $L_p^{\mathrm{alg}}(\Sym^2 f \otimes \psi^{-1} \varepsilon^{-2}, 2k-1-s)$. This has exactly the same form as the functional equation of the ``analytic'' $p$-adic $L$-function of Theorem \ref{thm:padicL}, and it interchanges components $\eta$ with $\eta(-1) = +1$ with those with $\eta(-1) = -1$. Hence it is sufficient to treat the Main Conjecture only for characters of one particular sign.
   \end{remark}
  
  
     \begin{remark}
      Having introduced all the necessary objects, we can now outline the strategy of the next sections. Our goal is to bound $H^1_{\Gr, 1}(\Qi, A)$ in terms of a $p$-adic $L$-function.
      
      The Kolyvagin--Rubin machine shows (see \cite[Theorem 2.3.3]{rubin00}) that if a non-trivial Euler system exists, then the index of the Euler system in $H^1_{\Iw}(\Qi, T) = H^1_{\Iw, \Gr, 0}(\Qi, T)$ gives a bound for the Selmer group of $A$ with the strict local condition at $p$, which is our $H^1_{\Gr, 3}(\Qi, A)$. 
      
      We shall in fact construct an Euler system with an additional local property -- namely, all our Euler system classes will lie in $H^1_{\Iw, \Gr, 1}$ -- and the refinement of the Kolyvagin--Rubin machinery developed in \cite[Appendix B]{leiloefflerzerbes14b} and \cite[\S 12]{KLZ1b} then bounds $H^1_{\Gr, 2}(\Qi, A)$ in terms of the index of the Euler system in $H^1_{\Iw, \Gr, 1}(\Qi, T)$.
      
      We then perform one final step, using Poitou--Tate duality, which will tell us that the Selmer group $H^1_{\Gr, 1}(\Qi, A)$ is bounded in terms of the index of the image of the Euler system in the quotient $H^1_{\Iw, \Gr, 1} / H^1_{\Iw, \Gr, 2}$. This last index is essentially the $p$-adic $L$-function, and this gives our main theorem.
     \end{remark}

  \subsection{Control theorems}
  
   We now study the relation between the Selmer groups over $\Qi$ and over $\QQ$. We continue to assume that $\varepsilon(p) \psi(p) \ne 1$. 
      
   \begin{proposition}
    The restriction map
     \[ 
      H^1_{\Gr, i}(\QQ, A(\tau^{-1})) \rTo^{\mathrm{res}} H^1_{\Gr, i}(\Qi, A)^{\Gamma=\tau}.
     \]
     has finite kernel and cokernel. If $H^0(\Qi, A) = 0$ then it is injective.
   \end{proposition}
   
   \begin{proof}
    This is an easy exercise using the inflation-restriction exact sequence.
   \end{proof}

   \begin{proposition}
    \label{prop:descent}
    Let $L_p^{\mathrm{alg}}$ be as defined above, and let $\tau$ be a character of $\Gamma$. Let $i = 1$ if $\tau(-1) = 1$, and $i = 2$ if $\tau(-1) = -1$. Then $H^1_{\Gr, i}(A(\tau))$ is finite if and only if $L_p^{\mathrm{alg}}(\Sym^2 f \otimes \psi, \tau) \ne 0$. 
    
    If, moreover, $H^0(\Qi, A) = 0$, then we have the formula
    \[ \# H^1_{\Gr, i}(\QQ, A(\tau)) \cdot \# H^0(\Qp, A(\tau) / \sF^i) = p^{cn}, \]
    where $c = [L_{\frP}(\tau): \Qp]$ and 
    \[ n = v_p\left(L_p^{\mathrm{alg}}(\Sym^2 f \otimes \psi, \tau)\right).\]
   \end{proposition}
   
   \begin{proof}
    The claim that $L_p^{\mathrm{alg}}$ vanishes if and only if $H^1_{\Gr, i}(\QQ, A(\tau))$ is infinite is clear from the previous proposition. To obtain the exact estimate for the order of the cohomology groups, we use Nekovar's theory of Selmer complexes, and the compatibility of Selmer complexes with derived tensor product; as the computation is virtually identical to \cite[Theorem 11.6.5(ii)]{KLZ1b} we shall not give the details here.
   \end{proof}
   
   \begin{corollary}
    Let $\tau$ be a character of the form $j + \chi$, with $\chi$ of finite order and $j \in \{1, \dots, 2k-2\}$, and suppose that $L(\Sym^2 f \otimes \psi \chi^{-1}, j)$ is a critical value. Then $\Hf(\QQ, A(j + \chi))$ is finite if and only if $L^{\mathrm{alg}}_p(\Sym^2 f \otimes \psi, j + \chi) \ne 0$. In particular, if the Main Conjecture holds, then $\Hf(\QQ, A(j + \chi))$ is finite for all such $j$ and $\chi$.
   \end{corollary}
   
   \begin{proof}
    Immediate from the previous proposition and Proposition \ref{prop:globalBKSel}.
   \end{proof} 
   
 \section{The Euler system}
  \label{sect:ES}

  We now construct an Euler system for the representation $T$, using the projections of the Beilinson--Flach classes introduced in \cite{leiloefflerzerbes14, KLZ1b}.

  \subsection{Beilinson--Flach classes}
 
   \begin{definition}
    For $m \ge 1$ coprime to $p$, and $c > 1$ coprime to $6pmN_f$, let
    \[ \cBF_{m}^{f, f} \in H^1_{\Iw}(\QQ(\mu_{mp^\infty}), M_{L_{\mathfrak{P}}}(f \otimes f)^*) \]
    be the Beilinson--Flach cohomology class associated to the pair $(f_{\alpha}, f_{\alpha})$.
   \end{definition}
   
   By construction, $M_{L_\frP}(f \otimes f)^*$ is a direct summand of the cohomology group 
   \[ H^1_{\et}\left(Y_1(N_f)_{\QQbar}, \TSym^{k-2}(\sH_{\Qp})(1)\right)^{\otimes 2} \otimes L_{\frP},\]
   and this direct summand is preserved by the symmetry involution $s$ which interchanges the two factors in the tensor product.
   
   The key input to all our constructions is the following symmetry property:
   
   \begin{proposition}
    We have
    \[ s\left( \cBF_{m}^{f, f} \right) = -[\sigma] \cdot \cBF_{m}^{f, f}, \]
    where $\sigma \in \Gal(\QQ(\mu_{mp^\infty}) / \QQ)$ is the complex conjugation.
   \end{proposition}
   
   \begin{proof}
    Let $\rho$ be the automorphism of the variety $Y_1(N_f)^2$ that interchanges the two factors. Since the K\"unneth isomorphism
    \[ H^2(Y_1(N_f)^2_{\QQbar}, \TSym^{k-2}(\sH_{\Qp})^{\otimes 2}(2)) \cong \left(H^1(Y_1(N_f)_{\QQbar}, \TSym^{k-2}(\sH_{\Qp})(1))\right)^{\otimes 2}\]
    is given by cup-product, which is anti-commutative in degree 1, the map $\rho^*$ induces $-s$ on $M_{L_\frP}(f \otimes f)^*$. On the other hand, \cite[Proposition 5.2.3 (1)]{KLZ1b} shows that $\rho^*\left(\cBF_{m}^{f, f}\right) = [\sigma] \cdot \cBF_{m}^{f, f}$. Combining these statements gives the result.
   \end{proof}
   
   The eigenspaces of $s$ give the direct sum decomposition
   \[ 
    M_{L_\frP}(f \otimes f)^* = 
    \Sym^2 M_{L_{\frP}}(f)^* \oplus \sideset{}{^2}{\textstyle\bigwedge} M_{L_{\frP}}(f)^*,
   \]
   so we obtain the following:
   
   \begin{corollary}
    Let $\chi$ be a continuous character of $\Gal(\QQ(\mu_{mp^\infty}) / \QQ) \cong (\ZZ / mp^\infty \ZZ)^\times$. Then the image of $\cBF_{m}^{f, f}$ in $H^1(\QQ(\mu_m)^+, M_{L_\frP}(f \otimes f)^*(\chi))$ takes values in the following direct summand:
    \begin{itemize}
     \item in $\bigwedge^2 M_{L_\frP}(f)^*(\chi)$ if $\chi(-1) = 1$;
     \item in $\Sym^2 M_{L_\frP}(f)^* (\chi)$ if $\chi(-1) = -1$.
    \end{itemize}
   \end{corollary}
   
   \begin{remark}
    In the case $k = 2$, this relation was used in \cite{Dasgupta-factorization} in a setting where $\chi(-1) = 1$, in order to force the Beilinson--Flach class to be a nonzero scalar multiple of a circular unit. In contrast, we shall focus in this paper on the case $\chi(-1) = -1$. Nonetheless we are able to make use of Dasgupta's results, because of the fortunate fact that the functional equation interchanges characters of sign $+1$ and those of sign $-1$.
   \end{remark}
   
   Recall that we defined $V = \Sym^2 M_{L_\frP}(f)^*(1)(\psi)$, for $\psi$ a Dirichlet character of conductor $N_\psi$, and $T$ the natural $\cO_{L, \frP}$-lattice in $V$. Assume that $c$ also satisfies $(c, N_\psi) = 1$. Let $\nu = \psi(-1)$.
   
   \begin{definition}
    Let $\cBF_{\psi}^{f} \in H^1_{\Iw}(\Qi, T)^\nu$ be the image of $\cBF_{N_\psi}^{f, f}$ under the natural twisting map
    \begin{align*}
     H^1_{\Iw}(\QQ(\mu_{N_\psi p^\infty}), \Sym^2 M_{L_\frP}(f)^*)^{(-1)} 
     &\rTo H^1_{\Iw}(\QQ(\mu_{p^\infty N_\psi}), V)^\nu \\
     & \rTo^{\mathrm{cores}} H^1_{\Iw}(\Qi, V)^\nu.
    \end{align*}
    (Here the superscript $\nu$ denotes the eigenspace for the action of the element $-1 \in (\ZZ / mp^\infty \ZZ)^\times$, for any $m \ge 1$. It is clear from the construction that this class takes values in the lattice $T$.)
   \end{definition}
   
   Crucially, this element extends to a compatible family:
   
   \begin{theorem}
    \label{thm:ES}
    Let $\mathcal{R}$ denote the set of positive square-free integers coprime to $6pc N_f N_\psi$. Then there exists a family of cohomology classes $\mathbf{c} = (c_m)_{m \in \cR}$, with
    \[ c_m \in H^1_{\Iw}(\QQ(\mu_{mp^\infty}), T)^\nu, \]
    such that $c_1 = \cBF^{f}_{\psi}$, and if $\ell$ is a prime such that $m \in \cR$ and $\ell m \in \cR$, then
    \[ \operatorname{cores}_m^{\ell m}(c_{m\ell}) = 
     P_\ell(f \otimes f \otimes \psi, \Frob_\ell^{-1}) c_{m}.
    \]
   \end{theorem}
   
   Here $P_\ell(f \otimes f \otimes \psi, X)$ is the local Euler factor at $\ell$ of the Rankin--Selberg convolution $L$-function, so that
   \[ P_\ell(f \otimes f \otimes \psi, X)  = (1 - \ell^{k-1} \psi\varepsilon(\ell) X) P_\ell(\Sym^2 f \otimes \psi, X).\]
   
   \begin{proof}
    This is simply a translation into our present context of \cite[Theorem 11.4.1]{KLZ1b}.
   \end{proof}
   
  \subsection{Regulator formulae}

   \begin{proposition}
    \label{prop:selmer}
    We may choose the classes $c_m$ in Theorem \ref{thm:ES} in such a way that for any $m \in \cR$, the localisation of the class $c_m$ at $p$ lies in 
    \[ H^1_{\Iw}(\QQ(\mu_{mp^\infty}) \otimes \Qp, \sF^1 T)^{\nu}.\]
    Equivalently, in the notation of Definition \ref{def:GreenbergSel} we have
    \[ c_m \in  H^1_{\Iw, \Gr, 1}(\QQ(\mu_{mp^\infty}), T)^{\nu}.\]
   \end{proposition}
   
   \begin{proof}
    This is a special case of \cite[Proposition 8.1.8]{KLZ1b}. We are regarding $T$ as a direct summand of the module $M_{\cO_{L, \frP}}(f \otimes f)^*(1 + \psi)$; and the top graded piece $\Gr^0 T = T / \sF^1 T$ projects isomorphically onto the quotient $\sF^{--} M_{\cO_{L, \frP}}(f \otimes f)^*(1 + \psi)$ in the notation of \emph{op.cit.}. So the cited proposition shows that the image of $c_m$ under projection to $\Gr^0 T$ is zero, as required.
   \end{proof}
   
   We now study the image of this class in Iwasawa cohomology of $\Gr^1 T = \sF^1 T / \sF^2$.
   
   \begin{definition}
    \[ \mathcal{L}: H^1_{\Iw}\left(\Qpi, \Gr^1 V \right)^\nu \rTo \Lambda(\Gamma)^\nu \otimes \Dcris(\Gr^1 V) \]
    be the Perrin-Riou big logarithm (or big dual exponential) map.
   \end{definition}
   
   \begin{remark}
    The usual definition of the Perrin-Riou logarithm for $\Gr^1 V$ would give a map taking values in distributions of order $k$, since $k$ is the Hodge--Tate weight of $\Gr^1 V$. This is undesirable, so here we are actually using the Perrin-Riou regulator for the unramified representation $\Gr^1 V(-k)$ twisted by $k$, which coincides with the regulator for $\Gr^1 V$ divided by a product of $k$ logarithm factors. Cf.~the discussion in \cite[Remark 8.2.9]{KLZ1b}.
    
    One also needs to make a small modification if $\varepsilon\psi(p) = 1$. In this case the logarithm does not take values in $\Lambda$ but rather in $J^{-1} \Lambda$, where $J$ is the ideal that is the kernel of evaluation at $s = k-1$. If $\varepsilon\psi$ is trivial (or merely even) then this is not an issue, as $s = k-1$ has sign $-\nu$, so $(J^{-1} \Lambda)^\nu = \Lambda^\nu$. If $\varepsilon\psi$ is odd, but contains $p$ in its kernel, then the map does genuinely have a pole at $s = k-1$; but in any case the element $\cBF_\psi^f$ will still land in $\Lambda(\Gamma)^\nu \otimes \Dcris(\Gr^1 V)$, by Theorem 8.2.3 of \cite{KLZ1b}, and besides this specialisation does not correspond to a critical value of the symmetric square $L$-function, so we can legitimately ignore it.
   \end{remark}
   
   \begin{definition}\mbox{~}
    \begin{enumerate}[(i)]
     \item Let $\omega_f \in \Fil^1 H^1_{\dR}(Y_1(N_f) / \overline{\QQ}, \Sym^{k-2} \sH_{\dR}^\vee)$ be the algebraic differential form attached to $f$, and let 
     \[ \omega_f' = G(\varepsilon^{-1}) \cdot \omega_f, \]
     so that $\omega_f'$ is an $L$-basis vector of $\Fil^1 \Dcris(M_{L_{\frP}}(f))$ (cf.~\cite[Lemma 6.1.1]{KLZ1a}).
     \item Let $\eta_f^\alpha$ be the vector in $\Dcris(M_{L_\frP}(f))^{\varphi = \alpha} \otimes \Qpbar$ defined in \cite[Proposition 6.1.3]{KLZ1a}, and let 
     \[ 
      \eta_f' = G(\varepsilon^{-1}) \cdot \eta_f^\alpha \in \Dcris(M_{L_{\frP}}(f))^{\varphi = \alpha}.
     \]
     \item Let
     \[ 
      \xi_f = \tfrac12\left( \eta_f' \otimes \omega_f' + \omega_f' \otimes \eta_f'\right) \in \Dcris\left(\Sym^2 M_{L_{\frP}}(f)\right),
     \]
     and
     \[ \xi_{f, \psi} = G(\psi^{-1}) \cdot \xi_f \in \Dcris\left(V^*(1))\right).\]
    \end{enumerate}
   \end{definition}
   
   Via the canonical pairing $\Dcris(V) \otimes \Dcris(V^*(1)) \to \Dcris(L_{\frP}(1)) \cong L_{\frP}$, we can identify $\xi_{f, \psi}$ with a linear functional on $\Dcris(V)$. It is easy to see that this functional gives an isomorphism
   \[ \Dcris(\Gr^1 V) \to L_{\frP}. \]
   
   \begin{theorem}
    \label{thm:explicitrecip}
    We have
    \begin{align*}
     \left\langle \mathcal{L}\left(\cBF^f_\psi\right), \xi_{f, \psi}\right\rangle &= (-1)^s (c^2 - c^{2s - 2k + 2}  \psi\varepsilon(c)^{-2})  G(\psi^{-1})^2 G(\varepsilon^{-1})^2 L_p^\imp(f, f, \psi, s)\\
     &= (-1)^s (c^2 - c^{2s - 2k + 2}  \psi\varepsilon(c)^{-2})  G(\psi^{-1})^2 G(\varepsilon^{-1})^2 \\&\qquad  \times L_p^\imp(\Sym^2 f, \psi, s) L_{p, N_f N_\psi}(\varepsilon\psi, s - k + 1).
    \end{align*}
    In particular, the image of $\cBF^f_\psi$ in $H^1_{\Iw}(\Qpi, \Gr^1 V)^\nu$ is non-torsion.
   \end{theorem}
   
   (Here $s$ denotes a point of $\Spec \Lambda(\Gamma)^\nu$, i.e.~a continuous character of $\Gamma$ with $(-1)^s = \nu$; $c^s$ is the value of the character $s$ at $c \in \Zp^\times$, and similarly for $(-1)^s$.)
   
   \begin{proof}
    The second equality is Dasgupta's factorisation formula, so it suffices to prove the first equality. 
    
    If $\psi = 1$, then this is a direct consequence of the explicit reciprocity law for Beilinson--Flach elements (Theorem B of \cite{KLZ1b}), applied with $f = g$. To include the case of general $\psi$, we shall reduce it to the explicit reciprocity law for $f$ and $f \otimes \psi$, using exactly the same method as was used for the complex $L$-function in \cite[\S 4.3]{leiloefflerzerbes14}. Consider the ``twisting operator'' 
    \[ R_{\psi} \coloneqq \sum_{a \in (\ZZ / N_\psi \ZZ)^*} \psi(a)^{-1} \left[ \begin{pmatrix} 1 & 0 \\ 0 & 1\end{pmatrix} \times \begin{pmatrix} 1 & a/N_\psi \\ 0 & 1\end{pmatrix}\right].\]
    This gives a correspondence from $Y_1(N_\psi^2 N_f)$ to $Y_1(N_f)$ over $\QQ(\mu_{N_\psi})$, and pushforward via this correspondence sends $\cBF_{1, N_\psi^2 N_f}$ to $\cBF_{\psi, N_f}$. Dually, pullback via this correspondence sends $\eta_f \otimes \omega_f$ to $G(\psi^{-1}) \cdot \eta_f \otimes \omega_{f_\psi}$, where $f_{\psi} = \sum_{(n, N_\psi) = 1} a_n(f) \psi(n) q^n$. (Note that $f_{\psi}$ is an eigenform, although it may not be new at the primes dividing $N_\psi$.) It is clear that
    \[ L^\imp_p(f, f_\psi, s) = L^\imp_p(f, f, \psi, s), \]
    and the result follows.
   \end{proof}
   
   \begin{remark}
    The construction of the Beilinson--Flach elements $\cBF^f_\psi$, and the proof of the reciprocity law, rely on analytic continuation from characters of the form $s + \chi$ with $1 \le s \le k-1$ (satisfying, of course, the parity condition $(-1)^s \chi(-1) = \psi(-1)$). In this range, the Dirichlet $L$-value $L(\psi \varepsilon, s - k + 1)$ is critical, while the symmetric square $L$-value $ L(\Sym^2 f \otimes \psi, s)$ vanishes to degree 1; so the existence of nontrivial motivic cohomology classes for the symmetric square is to be expected from Beilinson's formula.
   
    On the other hand, we shall now specialise this formula at characters of the form $s + \chi$ with $k \le s \le 2k-1$ (but still satisfying the same parity condition). In this case everything is swapped: the symmetric square $L$-value is now a critical value, and it is the Dirichlet $L$-value that is non-critical.
   \end{remark}

  \subsection{Removing junk factors}
  
   The following proposition will help us to remove some of the factors in Theorem \ref{thm:explicitrecip}.
   
   \begin{proposition}
    Let $j \in \{k, \dots, 2k-2\}$, and let $\chi$ be a finite order character of $\Gamma$ such that $(-1)^j \chi(-1) = \psi(-1)$, so $j + \chi \in \Spec \Lambda^\nu$.
    
    \begin{enumerate}
     \item If $(\varepsilon\psi)^2 \ne 1$, or if $j \ne k$, or if $\chi^2 \ne 1$, then we may find some integer $c > 1$ coprime to $6 p N_f N_\psi$ such that $c^2 - c^{2s - 2k + 2}\varepsilon\psi(c)^{-2}$ is non-vanishing at $s = j + \chi$.
          
     \item Suppose $\varepsilon\psi \ne 1$ and $j = k$. Then $L_{p, N_f N_\psi}(\varepsilon\psi, s - k + 1)$ is non-vanishing at $s = j + \chi$. If $\varepsilon\psi = 1$ then this holds as long as $\chi \ne 1$.

     \item Suppose $\varepsilon\psi = 1$. Then the $p$-adic zeta function $L_{p, N_f N_\psi}(\varepsilon\psi, s - k + 1) = \zeta_{p, N_f}(s - k + 1)$ has a simple pole at $s = k$, and for any integer $c > 1$ coprime to $p$, the product
     \[ (c^2 - c^{2s - 2k + 2}) \zeta_{p, N_f}(s - k + 1) \]
     is holomorphic and non-zero at $s = k$.
     
     \item Suppose $\varepsilon\psi = 1$ and either $j \ne k$ or $\chi \ne 1$. If either $p$ is a regular prime, or $j + \chi = k \bmod p-1$ (i.e.~the characters $j + \chi$ and $k$ lie in the same component of weight space), then $\zeta_{p, N_f}(s - k + 1)$ is non-vanishing at $s = j + \chi$.
    \end{enumerate}
   \end{proposition}
   
   \begin{proof}
    Part (1) is trivial. The remaining parts are simply standard facts about the $p$-adic zeta function. We note first that the finite product of local factors 
    \[ \frac{L_{p, N_f N_\psi}(\varepsilon\psi, s-k+1)}{L_p(\varepsilon \psi, s-k+1)} = \prod_{\substack{\ell \mid N_f N_\psi \\ \ell \nmid N_{\varepsilon\psi}}}(1 - \ell^{k-1-s} \varepsilon\psi(\ell))\]
    cannot vanish at $s = j + \chi$ unless $j \le k-1$, so it suffices to treat the primitive $p$-adic zeta functions.
    
    For (2), our hypotheses force the character $\tau = \chi^{-1} \varepsilon \psi$ to be a non-trivial even character. Hence Leopoldt's formula expresses $L_p(\varepsilon \psi, 1 + \chi)$ as a non-zero multiple of
    \[ \sum_{a \in (\ZZ / m \ZZ)^*} \tau(a) \log_p(1 - \zeta_m^a), \]
    where $m$ is the conductor of $\tau$. Since Leopoldt's conjecture is known to hold for abelian extensions of $\QQ$, this is non-zero.
    
    For parts (3) and (4), we note that $\zeta_p(s)$ is well known to have a simple zero at $s = 1$. Hence the function ${}_c \zeta_p(s) = (c^2 - c^{2s})\zeta_p(s)$ is holomorphic and nonzero at $s = 1$, which is (3). For (4), we note that ${}_c \zeta_p(s)$ is bounded in norm by 1 on the component of weight space containing $s = 1$; and since ${}_c \zeta_p(1) = 2 (1 - 1/p)\log_p(c)$, if we choose $c$ to be a topological generator of $\Zp^\times$, then the bound is attained. It then follows that ${}_c \zeta_p(s)$ is invertible on this component and hence non-vanishing there. Finally, if $p$ is a regular prime, then the $p$-adic zeta function takes unit values on all other components of weight space, and hence it cannot vanish anywhere.
   \end{proof}
   
   \begin{definition}
    In the setting of the previous proposition, we shall say that the character $j + \chi$ is \emph{unfortunate} if there is no integer $c > 1$ coprime to $6pN_f N_\psi$ such that $\left(c^2 - c^{2s - 2k + 2}\varepsilon\psi(c)^{-2}\right) L_{p, N_f N_\psi}(\varepsilon\psi, s - k + 1)$ is non-vanishing at $s = j + \chi$.
   \end{definition}
      
   \begin{remark}
    An annoying, but ultimately minor, issue is what happens when $j = k$ and $\varepsilon \psi$ and $\chi$ are both quadratic, but not both trivial. In this case, we cannot get rid of the $c$ factor, but nor can we use it profitably to cancel out a pole in the zeta function (because there is no such pole), so $k + \chi$ is unfortunate. (A conceptual interpretation of this phenomenon is that if we allowed $f$ to be of CM type, then the symmetric square of $f$ would be reducible, and an appropriate twist of $L(\Sym^2 f, s)$ would then genuinely have a pole at $s = k$ which could be cancelled by the $c$ factor.)
    
    A much more serious problem is that we do not know if $L_p(\chi, n) \ne 0$ for all integers $n \ge 2$, other than the easy case where $\chi = 1$ and $p$ is regular. This is one of the many equivalent forms of Schneider's conjecture \cite{schneider79}.
   \end{remark}
   
   \begin{corollary}
    Let $j$ be an integer with $k \le j \le 2k-2$, and $\chi$ a finite-order character of $\Gamma$ with $(-1)^j \chi(-1) = \psi(-1)$. Assume that $f$ is minimal among its twists, and that $j + \chi$ is not unfortunate. If we are in the exceptional-zero case (i.e. $\varepsilon\psi(p) = 1$, $j = k$, and $\chi = 1$), assume also that $L_p(\Sym^2 f, \psi, s)$ vanishes to order exactly 1 at $s = k$.
    
    Then the image of $\cBF_{\psi}^{f}$ in $H^1\left(\Qp, \Gr^1 V(-j-\chi)\right)$ is non-zero. 
   \end{corollary}
      
   \begin{proof}
    We first treat the non-exceptional case. Let $I_{j + \chi}$ denote the prime ideal of $\Lambda(\Gamma)$ which is the kernel of the character $j + \chi$. Then we have a commutative diagram of isomorphisms
    \begin{diagram}
     H^1_{\Iw}(\Qpi, \Gr^1 V)/ I_{j + \chi} &\rTo^{\mathcal{L}} & \Dcris(\Gr^1 V) \otimes \Lambda(\Gamma)/I_{j + \chi}\\
     \dTo & & \dTo\\
     H^1(\Qp, \Gr^1 V(-j-\chi)) & \rTo^{R \cdot \exp^*} & \Dcris(\Gr^1 V)
    \end{diagram}
    where the right-hand vertical arrow is evaluation at $j + \chi$, and $R$ is an explicit non-zero factor. (Note that all the spaces in this diagram are 1-dimensional.)
    
    Since $j + \chi$ is not unfortunate or exceptional, and the imprimitive $L$-function $L^\imp(\Sym^2 f, \psi, s)$ does not vanish for $\Re(s) \ge k$ by Remark \ref{rmk:nonvanishing}, we deduce that $\mathcal{L}(\cBF^f_\psi)(j + \chi) \ne 0$, i.e.~the image of $\cBF^f_\psi$ is non-zero in the bottom right corner of the diagram. Because all the maps in the diagram are isomorphisms, $\cBF^f_\psi$ cannot map to zero in the bottom left corner, as required.
    
    Now let us treat the exceptional case. In this situation, $\Gr^1 V(-k)$ is the trivial representation of $G_{\Qp}$, so the image of the Perrin-Riou regulator $\mathcal{L}$ is contained in $I_k$. We now erect the diagram
    \begin{diagram}
     H^1_{\Iw}(\Qpi, \Gr^1 V) / I_k &\rTo^{\mathcal{L}} & I_k / I_k^2\\
     \dTo & & \dTo\\
     H^1(\Qp, \Gr^1 V(-k)) & \rTo & L_\frP.
    \end{diagram}
    Now the right-hand vertical arrow is given by $\mu \mapsto \mu'(k)$, and the bottom horizontal arrow is somewhat messy to describe (however, it must exist, since the vertical arrows are both isomorphisms). Our assumption that the $L$-invariant is non-zero again implies that the Beilinson--Flach element is non-zero in the bottom right-hand corner, and we are done. 
   \end{proof}
   
   \begin{remark}
    By Greenberg's conjecture, proved in \cite{Dasgupta-factorization}, one knows that in the the exceptional-zero case the first derivative of  $L_p(\Sym^2 f, \psi, s)$ at $s = k$ is given by a (non-zero) classical $L$-value multiplied by an ``$L$-invariant'', which can be expressed as the deriviative of the $U_p$-eigenvalue with respect to the weight variable of the Hida family passing through $f$. However, we cannot at present rule out the possibility that this $L$-invariant might be 0.
   \end{remark}

  \subsection{The exceptional zero case}
  
   We now prove a refinement of Theorem \ref{thm:explicitrecip} in the ``exceptional zero'' setting. Here we will use the fact the $p$-stabilized eigenform $f_\alpha$ is the weight $k$ specialisation of some Hida family $\mathbf{f}$ -- a $p$-adic family of ordinary eigenforms of varying weight, indexed by some finite covering $\cW_\mathbf{f} = \Spec \Lambda_{\mathbf{f}}$ of a component of $\cW$.
   
   \begin{theorem}[Hida]
    Let $\tilde{L}_p(\mathbf{f}, \mathbf{f}, \psi)$ be the unique meromorphic function on $\cW_\mathbf{f} \times \cW_\mathbf{f}$ such that the 3-variable $p$-adic $L$-function $L_p(\mathbf{f}, \mathbf{f}, \psi)(k, k', s)$ for $s = k'$ factors as
    \[ L_p(\mathbf{f}, \mathbf{f}, \psi)(k, k', s)|_{s = k'} = \left(1 - \frac{\alpha(k')}{\alpha(k) \psi(p) \varepsilon(p)}\right) \cdot \tilde{L}_p(\mathbf{f}, \mathbf{f}, \psi)(k, k'). \]
   
    If $\varepsilon\psi$ is the trivial character, then for any integer point $k \ge 2 \in \cW_\mathbf{f}$, the meromorphic function $x \mapsto \tilde{L}_p(\mathbf{f}, \mathbf{f}, \psi)(k, x)$ on $\cW_\mathbf{f}$ has a simple pole at $x = k$. If $\varepsilon\psi$ is non-trivial, then this function is holomorphic at $x = k$.
   \end{theorem}
   
   We refer to $\tilde{L}_p(\mathbf{f}, \mathbf{f}, \psi)$ as the ``improved'' $p$-adic Rankin--Selberg $L$-function. We claim that this $L$-function is the image of a Beilinson--Flach class under a corresponding ``improved'' Perrin-Riou logarithm.
   
   \begin{theorem}[Venerucci]
    If $M$ is a $p$-torsion-free, pro-$p$ abelian group with a continuous, unramified action of $G_{\Qp}$, then there is an isomorphism
    \[ \widetilde\exp^*_{\Qp, M}: \frac{H^1(\Qp, M)}{H^1_{\mathrm{nr}}(\Qp, M)} \to \DD(M), \]
    where $\DD(M) = (M \mathop{\hat\otimes} \widehat{\ZZ}_p^{\mathrm{nr}})^{G_{\Qp}}$, such that
    \begin{itemize}
     \item the construction of $\widetilde\exp^*_{\Qp, M}$ is functorial in $M$;
     \item if $M$ is a finitely-generated $\Zp$-module, then we have
     \[ \widetilde\exp^*_{\Qp, M} = -(1 - p^{-1} \varphi^{-1})^{-1} \exp^*_{\Qp, M} \]
     where $\exp^*_{\Qp, M}$ is the usual Bloch--Kato dual exponential map for $M[1/p]$;
     \item we have
     \[ \mathcal{L}_{M}(0) = -(1 - \varphi) \widetilde\exp^*_{\Qp, M},\]
     where $0$ denotes the trivial character of $\Gamma$.
    \end{itemize}
   \end{theorem}
   
   \begin{proof}
    See \cite[Proposition 3.8]{venerucci16}. (The argument is given there for a specific module $M$ arising from a Hida family, but the same argument works in full generality.)
   \end{proof}
   
   \begin{theorem}
    We have
    \begin{multline*}
     \left\langle \widetilde\exp^*_{\Gr^1 V(-k)}\left( \cBF^f_\psi\right), \xi_{f, \psi} \right\rangle \\= (-1)^{k + 1} \frac{G(\varepsilon^{-1})^2 G(\psi^{-1})^2}{(1 - \tfrac\beta\alpha)(1 - \tfrac{\beta}{p\alpha})} \cdot \left( \lim_{x \to k}(c^2 - c^{2 + k - x}\varepsilon(c)^{-2}\psi(c)^{-2})   \tilde L_p(\mathbf{f}, \mathbf{f}, \psi)(k, x)\right).
    \end{multline*}
    In particular, if $\psi = \varepsilon^{-1}$ then the image of $\cBF^f_\psi$ in $H^1(\Qp, V(-k))$ does not lie in $\Hf$.
   \end{theorem}
   
   \begin{proof}
    If $\varepsilon(p)\psi(p) \ne 1$, so the Euler factor $\left(1 - \frac{\alpha(k')}{\alpha(k) \psi(p) \varepsilon(p)}\right)$ does not vanish at $s = k$, this is a restatement of Theorem \ref{thm:explicitrecip}. In the bad cases, the result is obtained by applying Theorem B of \cite{KLZ1b} (modified as above when $\psi \ne 1$) to obtain an equality of meromorphic functions on $\mathcal{W}_\mathbf{f} \times \mathcal{W}_\mathbf{f}$, cancelling the Euler factor (since it is invertible almost everywhere), and taking the limit as $x \to k$.
    
    The final statement follows from the fact that when $\psi = \varepsilon^{-1}$, the simple pole of $\tilde L_p$ cancels with the zero of $c^2 - c^{2 + k - x}$ to give a non-vanishing value at $x = k$. So $\widetilde\exp^*(\cBF^f_\psi)$ is non-zero, implying that $\cBF^f_\psi$ does not lie in $\Hf$.
   \end{proof}
   
   \begin{remark}
    If $\varepsilon\psi$ is not the trivial character, but maps $p$ to 1, then we conclude that (for a suitable choice of $c$) we have $\widetilde\exp^*_{\Gr^1 V(-k)}\left( \cBF^f_\psi\right) \ne 0$ if and only if $\tilde L_p(\mathbf{f}, \mathbf{f}, \psi)(k, k) \ne 0$. 
    
    It seems likely that $\tilde L_p(\mathbf{f}, \mathbf{f}, \psi)(k, k)$ is never zero, and it may be possible to prove this using an adaptation of the methods of \cite{Dasgupta-factorization}. We hope to return to this issue in a future paper.
   \end{remark}

  
 \section{Bounding the Selmer group}
  \label{sect:kolyvagin}
  
  \subsection{Setup}
  
   Let $V = \Sym^2 M_{L_\frP}(f)^* (1 + \psi)$, as before, and let $V' = \bigwedge^2  M_{L_\frP}(f)^* (1 + \psi)$, so that $M_{L_\frP}(f \otimes f)^*(1 + \psi) = V \oplus V'$. Let $T$ and $T'$ be the $\cO_{L, \frP}$-lattices in $V$ and $V'$ induced by $ M_{\cO_{L, \frP}}(f)^*$.
   
   We shall impose the following hypotheses:
   \begin{itemize}
    \item $p \ge 7$;
    \item $\frP$ is a degree 1 prime of $L$, so $L_{\frP} \cong \Qp$;
    \item the image of $G_{\QQ}$ in $\Aut M_{L_\frP}(f) \cong \operatorname{GL}_2(\Qp)$ contains a conjugate of $\SL_2(\Zp)$;
    \item there exists $u \in (\ZZ / N_f N_\psi \ZZ)^\times$ such that $\varepsilon\psi(u) \ne \pm 1 \bmod \frP$ and $\psi(u)$ is a square modulo $\frP$;
    \item $\varepsilon\psi(p) \ne 1$.
   \end{itemize}
   
  \subsection{Verifying the hypotheses} We now collect together some technical facts about the Galois action on $T$ that will be needed in the Euler system argument.
   
   \begin{proposition}
    Under the above hypotheses, there exists $\tau \in G_{\QQ}$ such that:
    \begin{itemize}
     \item $\tau$ acts trivially on $\mu_{p^\infty}$,
     \item $T / (\tau - 1)T$ is free of rank 1 over $\Zp$,
     \item $\tau - 1$ acts invertibly on $T'$.
    \end{itemize}
   \end{proposition}
   
   \begin{proof}
    Let us choose $t \in (G_{\QQ})^{\mathrm{ab}}$ mapping to $u$ under the mod $N_f N_\psi$ cyclotomic character. Since $p \nmid N_f N_\psi$, we may assume that $\kappa(t) = 1$, where $\kappa$ is the $p$-adic cyclotomic character.
    
    Since the derived subgroup of $\GL_2(\Zp)$ is $\operatorname{SL}_2(\Zp)$, we see that the preimage of $t$ in $G_{\QQ}$ must contain elements acting on $M_{\cO_{L, \frP}}(f)^*$ by \emph{any} given matrix having determinant $\varepsilon(u)$. In particular, we may find an element $\tau$ which acts on $M_{\cO_{L, \frP}}(f)^*$ by a conjugate of 
    \(
     \begin{pmatrix} 
     \varepsilon(u) \psi(u)^{1/2} & 0 \\ 
      0 & \psi(u)^{-1/2}
     \end{pmatrix}
    \). 
    Hence $\tau$ acts on $T$ via a conjugate of $\begin{pmatrix} \varepsilon^2 \psi^2(u) & \\ & \varepsilon\psi(u) \\  & & 1\end{pmatrix}$, and on $T'$ as $\varepsilon\psi(u)$; so we are done.
   \end{proof}
   
   We also have the following lemma. For $r \ge 1$, 
   
   \begin{lemma}
    Let $B$ denote the image of $\SL_2(\Zp)$ in $\SL_3(\Zp)$ via the adjoint representation (or the symmetric square, since these coincide as representations of $\SL_2$). Suppose $p \ge 7$. Then we have
    \[ H^1\left(B, \mathbf{F}_p^3\right) = 0.\]
   \end{lemma}
   
   \begin{proof}
    This follows exactly as in \cite[Lemma 1.2]{flach92}.
   \end{proof}
   
   Now let $\Omega$ be the smallest extension of $\QQ$ such that $G_\Omega$ acts trivially on $T$, on $T'$, and on $\Zp(1)$.
   
   \begin{proposition}\label{prop:trivomegacoh}
    We have 
    \[ H^1(\Gal(\Omega / \QQ), T \otimes \FF_p) = H^1(\Gal(\Omega / \QQ), T^*(1) \otimes \FF_p) = 0.\]
   \end{proposition}
   
   \begin{proof}
    We give the proof for $T$ (the argument for $T^*(1)$ being essentially identical). Let $D$ be the image of $G_{\QQ}$ in $\Aut T \cong \GL_3(\Zp)$.
    
    By our assumptions on $\frP$, the image of $G_\QQ$ in $\Aut M_{L_\frP}(f)$ contains a conjugate of $\SL_2(\Zp)$, so the image of the commutator subgroup $(G_{\QQ})^{\mathrm{der}}$ must be conjugate to $\SL_2(\Zp)$. Hence the group $D$ contains, as a normal subgroup, a copy of the group $B$ from the preceding proposition. Therefore, in the exact sequence
    \[ H^1(D/B, H^0(B, \FF_p^3)) \rTo^{\operatorname{inf}} H^1(D, \FF_p^3) \rTo^{\mathrm{res}} H^1(H, \FF_p^3), \]
    the first and last terms are zero; so $H^1(D, \FF_p^3) = 0$.
    
    If $\Omega_T \subseteq \Omega$ is the splitting field of $T$, then $\Gal(\Omega_T / \QQ)$ maps isomorphically to $D$, so we have another exact sequence
    \[ 
     H^1(D, \FF_p^3) \rTo^{\operatorname{inf}} H^1(\Gal(\Omega / \QQ), \FF_p^3) \rTo^{\mathrm{res}} \Hom_D(\Gal(\Omega / \Omega_T), \FF_p^3).
    \]
    The first term is 0, as we have just seen; and in the third term, the action of $D$ on $\Gal(\Omega / \Omega_T)$ is by conjugation, but this action is trivial, since $\Omega$ is the composite of $\Omega_T$ and an abelian extension of $\QQ$. Since $(\FF_p^3)^D$ is evidently zero, we conclude that $H^1(\Gal(\Omega / \QQ), \FF_p^3)$ is zero as required.
   \end{proof}

      
  \subsection{Modifying the Euler system}
  
   We continue to assume the hypotheses of the previous section.

   \begin{lemma}
    For any prime $\ell \nmid pcN_f N_\psi$, any Dirichlet character $\eta$ of $p$-power conductor, and any $k \ge 0$, the map $\Frob_\ell^{p^k} - 1$ is injective on $T(\eta^{-1})$.
   \end{lemma}
   
   \begin{proof}
    This is clear from the fact that the eigenvalues of $\Frob_\ell$ on $T$ are Weil numbers of weight $2k > 0$.
   \end{proof}
   
   \begin{definition}
    Let $\mathcal{P}'$ be the set of primes $\ell \nmid pcN_f N_\psi$ such that
    \begin{itemize}
     \item $\ell = 1 \bmod p$,
     \item $T / (\Frob_\ell - 1) T$ is a cyclic $\Zp$-module,
     \item $\Frob_\ell - 1$ is bijective on $T'$.
    \end{itemize}
    Let $\mathcal{R}'$ be the set of square-free products of primes in $\mathcal{P}'$, and for $m \in \mathcal{R}'$, let $\QQ(m)$ be the unique abelian $p$-extension of $\QQ$ contained in $\QQ(\mu_m)$.
   \end{definition}
   
   (Note that $\mathcal{P}'$ is not empty: if $\tau$ is the element of $G_\QQ$ from the previous section, and $Y$ is a finite extension of $\QQ$ that acts trivially on $T/pT$, $T'/pT'$ and $\mu_p$, then any prime $\ell$ whose Frobenius in $\Gal(Y / \QQ)$ is conjugate to $\tau$ will lie in $\mathcal{P}$.)
   
   \begin{theorem}
    There exists a collection of classes $\mathbf{c}' = (c_m')_{m \in \mathcal{R}'}$, where
    \[ c_m' \in H^1_{\Iw}(\QQ(m)(\mu_{p^\infty}), T)^\nu, \]
    such that
    \begin{enumerate}[(i)]
     \item $c_1' = \cBF^f_\psi$;
     \item $c_m' \in H^1_{\Iw, \Gr, 1}(\QQ(m)(\mu_{p^\infty}), T)^\nu$ for all $m$;
     \item if $m \in \mathcal{R}'$ and $\ell$ is a prime with $\ell m \in \mathcal{R}'$, then 
     \[ 
       \operatorname{cores}_m^{\ell m}(c_{\ell m}') = 
       P_\ell(\Sym^2 f \otimes \psi, \Frob_\ell^{-1}) c_m'.
     \]
    \end{enumerate}
   \end{theorem}
   
   \begin{proof}
    We shall construct the classes $c_m'$ by modifying the $c_m$ of Theorem \ref{thm:ES} appropriately. For a first approximation we consider the classes $c_m''$ defined as the projections of the $c_m$ from $\QQ(\mu_{mp^\infty})$ to its subfield $\QQ(m)(\mu_{p^\infty})$. These classes already satisfy (i) (by definition), and (ii) (by Proposition \ref{prop:selmer}); so we must modify them to obtain (iii), getting rid of the extra degree 1 factor in the norm relations.
    
    For $m \in \mathcal{R}$, let $\Gamma_m = \Gal(\QQ(m)(\mu_{p^\infty}) / \QQ)$. For each prime $\ell \mid m$, let $F_\ell \in \mathcal{R}'$ denote the unique element of $\Gamma_m$ which acts trivially on $\QQ(\ell)$ and maps to $\Frob_\ell$ in $\Gamma_{m/\ell}$. Then $1 - \ell^{k-1} \varepsilon\psi(\ell) F_\ell^{-1}$ is invertible in $\Zp[[\Gamma_m]]$ (because $\ell = 1 \bmod p$, so $F_\ell = 1$ modulo the radical of $\Zp[[\Gamma_m]]$).
    
    We now take
    \[ c_m' = \left( \prod_{\ell \mid m} (1 - \ell^{k-1} \varepsilon\psi(\ell) F_\ell^{-1})^{-1}\right) c_m''.\]
    It is clear that these classes have property (iii). Property (i) is obvious, since $c_1' = c_1$ by definition. Finally, property (ii) is safe, since the Greenberg Selmer groups are $\Zp[[\Gamma_m]]$-modules.
   \end{proof}
   
   We are now in a position to apply the theory of \cite[\S 12]{KLZ1b}. We fix a character $\eta$ of $\Gamma_{\mathrm{tors}}$, and let $\bT = T(\eta^{-1}) \otimes \Lambda(\Gamma_1)(-\mathrm{can})$, where $\Lambda(\Gamma_1)$ is the Iwasawa algebra of the cyclotomic $\Zp$-extension and ``can'' denotes the canonical character $G_{\QQ} \to \Lambda_{\cO}(\Gamma_1)^\times$. We equip $\bT$ with the Greenberg Selmer structure $\cF_{\Gr, 1}$, for which the local condition at $p$ is given by the cohomology of $\sF^1 \bT$, as in the previous sections.
   
   \begin{theorem}
    There is a generalized Kolyvagin system
    \[ \pmb{\kappa} \in \overline{\mathbf{KS}}(\bT, \cF_{\Gr, 1}, \cP'),\]
    in the notation of \cite{mazurrubin04}, such that $\mathbf{\kappa}_1$ is the $\eta$-isotypical projection of $\cBF^f_{\psi}$.
   \end{theorem}
   
   \begin{proof}
    This follows by applying Theorem 5.3.3 of \cite{mazurrubin04} to the collection $(c_m')_{m \in \cR'}$. This gives a generalized Kolyvagin system for the relaxed local condition at $p$, but since our $c_m'$ are in $H^1_{\Iw, \Gr, 1}$, the resulting Kolyvagin classes also satisfy the Greenberg local condition at $p$, as shown in \cite[\S 12]{KLZ1b}.
   \end{proof}
   
   It is important to note that, although the hypotheses (H.0)--(H.4) of \cite[\S 3.5]{mazurrubin04} are satisfied for the module $\bT$, the hypothesis (H.5) is not, since we need to exclude from $\cP'$ all primes for which $\Frob_\ell - 1$ does not act invertibly on $T'$. However, we will be choosing our auxilliary primes to have Frobenii close to $\tau$, and $\tau$ was selected so that $\tau -1$ acts bijectively on $T'$, so this is not a problem; the necessary modification of Mazur and Rubin's arguments is sketched in Appendix \ref{appendix}. 
  
  \subsection{Selmer bounds over \texorpdfstring{$\Qi$}{Qinf}}
  
   \begin{theorem}
    Let $\eta$ be a character of $\Gamma_{\mathrm{tors}}$ with $\eta(-1) = \psi(-1)$. Then:
    \begin{itemize}
     \item The Greenberg Selmer group
     \[ e_\eta H^1_{\Iw, \Gr, 1}(\Qi, T) \coloneqq e_{\eta} \ker\Big( H^1_{\Iw}(\Qi, T) \to H^1_{\Iw}(\Qpi, T / \sF^1 T)\Big) \]
     of $T$ is free of rank 1 over $e_\eta \Lambda_{\cO}$, and $e_\eta \cdot \cBF^f_\psi$ is a non-torsion element of this module.
     \item The Greenberg Selmer group
     \[ e_{\eta} \cdot H^1_{\Gr, 2}(\Qi, A)^\vee\]
     of $A$ is a torsion $e_\eta \Lambda$-module.
     \item We have the divisibility of characteristic ideals
     \[ \Char_{\Lambda}\left(e_\eta \cdot  H^1_{\Gr, 2}(\Qi, A)^\vee\right) \mid \Char_{\Lambda} \left(e_\eta \cdot \frac{H^1_{\Iw, \Gr, 1}(\Qi, T)}{\Lambda \cdot \cBF^f_{\psi}}\right).\]
    \end{itemize}
   \end{theorem}
   
   \begin{proof}
    See \cite[Corollary 12.3.5]{KLZ1b} (and the appendix below).
   \end{proof}
   
   We shall now derive from this a bound for the group $H^1_{\Gr, 1}(\Qi, A)^\vee$.
   
   \begin{theorem}
    \label{mainthm}
    \mbox{~}
    \begin{itemize}
     \item The group $e_\eta \cdot H^1_{\Iw, \Gr, 2}(\Qi, T)$ is zero. 
     \item The group $e_{\eta} \cdot H^1_{\Gr, 1}(\Qi, A)^\vee$ is a torsion $e_\eta \Lambda$-module.
     \item We have the divisibility of characteristic ideals
     \[ \Char_{\Lambda}\left(e_\eta \cdot H^1_{\Gr, 1}(\Qi, A)^\vee\right) \mid d \cdot e_\eta \cdot L_p^{\imp}(\Sym^2 f, \psi, s) \cdot L_{p, N_f N_\psi}(\varepsilon \psi, s - k + 1), \]
     where $d$ is any generator of the congruence ideal of $f$.
    \end{itemize}
   \end{theorem}
   
   \begin{proof}
    This follows from the previous theorem via the Poitou--Tate duality exact sequence
    \begin{multline*}
     0 \rTo H^1_{\Iw, \Gr, 2}(\Qi, T) \to H^1_{\Iw, \Gr, 1}(\Qi, T) \to H^1_{\Iw}(\Qpi, \Gr^1 T) \\ \to H^1_{\Gr, 1}(\Qi, A)^\vee \to H^1_{\Gr, 2}(\Qi, A)^\vee \to 0,
    \end{multline*}
    using the fact that
    \[ \Char_{\Lambda} \left(\frac{ H^1_{\Iw}(\Qi, \Gr^1 T)}{\operatorname{image} \cBF^f_{\psi}}\right) = \Char_{\Lambda}\left( \frac{d^{-1} \Lambda}{L_p^\mathrm{imp}(f, f, \psi)}\right)\]
    by the explicit reciprocity law.
   \end{proof}
   
   \begin{remark}
    Even when our rather strong hypotheses are satisfied, this result is ``non-optimal'', in two ways. Firstly, it involves the imprimitive $L$-function, rather than the primitive one; and secondly, there is the unwanted factor of the $p$-adic Dirichlet $L$-function.
   \end{remark}
   
   \begin{corollary}  
    Let $j \in \{k, \dots, 2k-2\}$ be an integer with $(-1)^j = \psi(-1)$, so that $L(\Sym^2 f, \psi, j)$ is a non-zero critical value. If $j > k$, then assume that the $p$-adic $L$-value $L_p(\varepsilon\psi, j - k + 1)$ is non-zero. Then $\Hf(\QQ, T(-j)) = 0$ and $\Hf(\QQ, A(j))$ is finite.
   \end{corollary}
   
   \begin{proof}
    Since our running assumptions certainly imply that $\varepsilon\psi$ is not trivial or quadratic, and $\varepsilon(p) \psi(p) \ne 1$, the hypotheses of the corollary imply that the product of $p$-adic $L$-functions on the right-hand side of Theorem \ref{mainthm} is non-zero at $s = j$. By the control theorem, we deduce that $H^1_{\Gr, 2}(\QQ, T) = 0$ and $H^1_{\Gr, 1}(\QQ, A(j))$ is finite; and these Greenberg Selmer groups coincide with the Bloch--Kato Selmer groups by Proposition \ref{prop:globalBKSel}.
   \end{proof}
     
   
 \section{A numerical example}
  \label{sect:example}
  
  To illustrate the results above, we give an example, computed using Dokchitser's \textsc{ComputeL} software for computing $L$-functions \cite{dokchitser04}. (We are grateful for a post on a web forum by Martin Raum, illustrating how to use Dokchitser's software for computing symmetric square $L$-functions.)
  
  \subsection{Computing the L-value}
   
   Let $F$ be the normalised eigenform of level 1 and weight 16. Note that the image of the mod $p$ Galois representation of $f$ contains a conjugate of $\SL_2(\Zp)$ for all primes $p$ except $\{2, 3, 5, 7, 11, 31, 59, 3617\}$ \cite{swinnertondyer72}, and the congruence ideal of $F$ is generated by $3617$. Since $F$ has level 1, we have $L(\Sym^2 f \otimes \psi, s) = L^{\mathrm{imp}}(\Sym^2 F, \psi, s)$ for any Dirichlet character $\psi$.
  
   Using Tim Dokchitser's $L$-function calculator and the relation
   \[ L(\Sym^2 f, k) = \frac{2^{2k-1} \pi^{k+1}}{(k-1)!} \langle f, f \rangle, \]
   (which is valid for any cuspidal level 1 eigenform of weight $k \in 2\ZZ$), one computes that
   \[ \langle F, F \rangle \cong 0.00000216906134759\dots \]
   
   We take for the character $\psi$ one of the two order 3 characters of $(\ZZ / 7 \ZZ)^\times$; we choose $\psi$ such that $\psi(3) = e^{2\pi i / 3}$. Since $\psi(-1) = 1$, the $L$-values $L(\Sym^2 F \otimes \psi, s)$ are critical for even integers $s$ in the range $16 \le s \le 30$. Using Dokchitser's software we computed the ratio
   \[ \tilde L(\Sym^2 F \otimes \psi, s) \coloneqq \frac{(s - 1)! (s - 16)! G(\psi^{-1})^2}{2^{2s + 1} \pi^{2s-15} \langle F, F \rangle} L(\Sym^2 F \otimes \psi, s) \]
   for $s$ in this range; Sturm's rationality theorem implies that these ratios should lie in $\QQ(\sqrt{-3})$. 
   
   We found that for $s = 22$ this ratio is $7.8323\dots + i \cdot 10.2324\dots$, and it agrees to over 250 decimal places\footnote{This computation required the calculation of the first 15000 $q$-expansion coefficients of $F$, and took around three minutes of CPU time.} with
   \[ \frac{136547867422656337144320 + 102994007489228654461440\sqrt{-3}}{17433892055631543710491}.\]
   
   \begin{remark}
    In principle this computation can be done exactly, by following the same steps used to prove the rationality of the $L$-values; but we have not attempted to do this.
   \end{remark}
   
   Let us assume that this approximate computation is in fact correct. The denominator of $\tilde L(\Sym^2 F \otimes \psi, 22)$ is a unit outside $\{ 7, 13 \}$ (note that 13 is a non-ordinary prime for $f$). The numerator is divisible only by primes of $\QQ(\sqrt{-3})$ above the following rational primes:
   \[ \{2, 3, 5, 43, 67, 103, 141264461964750634089522953623\}. \]
   
  \subsection{The case \texorpdfstring{$p = 37$}{p = 37}}
  
   If we take $p = 37$, for instance, and choose either of the two embeddings of $\QQ(\sqrt{-3})$ into $\Qp$, then $\tilde L(\Sym^2 F \otimes \psi, 22)$ is a $p$-adic unit. Moreover, $F$ is ordinary at $p$ and we have 
   \[ \alpha_p = 11 + 7 \times 37 + 25 \times 37^2 + \dots.\]
   The factor $\mathcal{E}'_p(s, \mathrm{id})$ appearing in Theorem \ref{thm:padicL} is evidently a unit for $s = 22$ (indeed it is congruent to 1 modulo $p^6$), so we deduce that $L_p(\Sym^2 F \otimes \psi, 22)$ is a $p$-adic unit. There are no exceptional-zero phenomena, since $\psi(p) \ne 1$. Finally, the $p$-adic zeta value $L_p(\psi, 7)$ will be congruent modulo $p$ to 
   \[ L_p(\psi, 7 - (p-1)) = L_p(\psi, -29) = (1 - p^{29} \psi(p))\left(\frac{-B_{30, \psi}}{30}\right) = 12 + 36 \times 37 + 23 \times 37^2 + \dots.\]
   which is in $\Zp^\times$. From Theorem \ref{mainthm}, we can now deduce that 
   \[ H^1_{\Gr, 1}(\QQ^{\mathrm{cyc}}, \Sym^2 M_{\Zp}(F)(\psi)(22) \otimes \Qp/\Zp) \]
   is pseudo-null, where $\QQ^{\mathrm{cyc}}$ is the cyclotomic $\ZZ_{37}$-extension of $\QQ$; that is, the $22$-nd isotypical component of the algebraic $p$-adic $L$-function is a unit.
   
   The strong form of the control theorem given in Proposition \ref{prop:descent}, together with Proposition \ref{prop:globalBKSel}, now shows that
   \[ \Hf(\QQ, \Sym^2 M_{\Zp}(F)(\psi)(22) \otimes \Qp/\Zp) = 0.\]
 
  \subsection{The case \texorpdfstring{$p = 67$}{p = 67}}
  
   The case $p = 67$ is more interesting, since if we embed $\QQ(\sqrt{-3})$ in $\Qp$ by using the prime $8 + \sqrt{-3}$, then $\tilde L(\Sym^2 F \otimes \psi, 22)$ is not a $p$-adic unit: it has valuation 1 at $p$. Arguing as before (and using the fact that $B_{60, \psi}$ is a 67-adic unit) we deduce that $\Hf(\QQ, \Sym^2 M_{\Zp}(F)(\psi)(22) \otimes \Qp/\Zp)$ is either trivial, or cyclic of order $p$. The main conjecture predicts that this group has order exactly $p$, but we cannot prove this by our methods.
   
  \subsection{The case \texorpdfstring{$p = 439$}{p = 439}}
  
   For $p = 439$ an interesting phenomenon occurs. In this case, if we take the embedding $\QQ(\sqrt{-3}) \into \Qp$ corresponding to the ideal $\frP = 14 + 9 \sqrt{-3}$, then $L_p(\Sym^2 F \otimes \psi, 22) \in \Zp^\times$, but the parasitic Dirichlet $L$-value $L_p(\psi, 7)$ appearing in Theorem \ref{mainthm} is not a unit (it is $148 \times 439 + 232 \times 439^2 + \dots$). Thus the main conjecture predicts that $\Hf(\QQ, \Sym^2 M_{\Zp}(F)(\psi)(22) \otimes \Qp/\Zp) = 0$ should be zero, but our methods can only prove that its order is at most $p$.
  
\appendix

 \section{Kolyvagin systems for direct sums}
  \label{appendix}
 
  In this section we prove a simple modification of a crucial lemma from \cite{mazurrubin04}, in order to permit us to work with Euler systems for reducible Galois representations (and their associated Kolyvagin systems). Our aim is to show a precise form of the following statement:
  
  \begin{quotation}
   ``A Kolyvagin system for a direct sum $T_1 \oplus T_2$ that happens to take values in $T_1$ is almost as good as a Kolyvagin system for $T_1$''.
  \end{quotation}
  
  \subsection{Setup}
  
   In this section we will use the following notations:
   \begin{itemize}
    \item $p$ is a prime, 
    \item $R$ is a complete Noetherian local ring, with finite residue field $\bk$ of characteristic $p$,
    \item $\frm$ is the maximal ideal of $R$.
   \end{itemize}
   
   We are interested in Kolyvagin systems for modules of the form $T_1 \oplus T_2$, where $T_i$ are $R[G_{\QQ}]$-modules, both free of finite rank over $R$. We let $\Sigma$ be a finite set of places of $\QQ$ containing $p$, $\infty$, and all the primes at which either $T_1$ or $T_2$ is ramified; and we choose a \emph{Selmer structure}\footnote{In the sense of \cite{mazurrubin04}, i.e.~a simple Selmer structure in the sense of \cite{KLZ1b}.} $\cF$ for $T_1$ with $\Sigma(\cF) = \Sigma$. 
   
   For $t \ge 1$, we define the following set of primes (cf.~\cite[Definition 3.1.6]{mazurrubin04}): $\cP_t^\sharp$ is the set of all primes $\ell \notin \Sigma$ such that
   \begin{itemize}
    \item $\ell = 1 \bmod \mathfrak{m}^k \cap \ZZ$,
    \item $T_1 / (\mathfrak{m}^k T_1 + (\Frob_\ell - 1) T_1)$ is free of rank 1 over $R / \mathfrak{m}^k$,
    \item $\Frob_\ell - 1$ acts bijectively on $T_2$.
   \end{itemize}
   It is clear that $\cP_1^\sharp \supseteq \cP_2^\sharp \supseteq \cP_3^\sharp$ and so on; and if $\frm^k = 0$, then $\cP_t^\sharp = \cP_k^\sharp$ for $t \ge k$. Finally, we let $\cP$ be a set of primes disjoint from $\Sigma$.
   
   If $T$ is any $\Zp[G_\QQ]$-module, we shall write $T^\vee = \Hom(T, \mu_{p^\infty})$.
   
  \subsection{Hypotheses}
  
   We shall impose a set of hypotheses (H.$0^\sharp$)--(H.$6^\sharp$) on the collection $(T_1, T_2, \cF, \cP)$. These are slight adaptations of the hypotheses (H.0)--(H.6) of \cite[\S 3.5]{mazurrubin04}.
   
   \begin{itemize}
    \item (H.$0^\sharp$) The $T_i$ are free $R$-modules.
    \item (H.$1^\sharp$) $T_1 / \mathfrak{m} T_1$ is an absolutely irreducible $\bk$-representation.
    \item (H.$2^\sharp$) There is a $\tau \in G_{\QQ}$ such that $\tau = 1$ on $\mu_{p^\infty}$, $T_1 / (\tau - 1) T_1$ is free of rank 1 over $R$, and $\tau - 1$ is bijective on $T_2$.
    \item (H.$3^\sharp$) If $\Omega = \QQ(T_1, T_2, \mu_{p^\infty})$ is the smallest extension of $\QQ$ acting trivially on $T_1$, $T_2$, and $\mu_{p^\infty}$, then $H^1(\Omega / \QQ, T_1 / \mathfrak{m} T_1) = H^1(\Omega / \QQ, T_1^\vee[\mathfrak{m}]) = 0$.
    \item (H.$4^\sharp$) Either:
    \begin{itemize}
     \item (H.4a$^\sharp$) $\Hom_{\FF_p[G_\QQ]}(T_1/\mathfrak{m} T_1, T_1^\vee[\mathfrak{m}]) = 0$, or
     \item (H.4b$^\sharp$) $p \ge 5$.
    \end{itemize}
    \item (H.$5^\sharp$) We have $\cP_1^\sharp \supseteq \cP \supseteq \cP_t^\sharp$ for some $t$.
    \item (H.$6^\sharp$) For every $\ell \in \Sigma$, the local condition $\cF$ at $\ell$ is Cartesian on the category of quotients of $R$ in the sense of \cite[Definition 1.1.4]{mazurrubin04}.
   \end{itemize}
   
   Note that hypotheses (H.$1^\sharp$), (H.$4^\sharp$), (H.$5^\sharp$) and (H.$6^\sharp$) are identical to Mazur and Rubin's hypotheses (H.i) for $T = T_1$; it is only (H.$2^\sharp$) and (H.$3^\sharp$) which are different. Note, also, that if $T_2 = \{0\}$ then all our hypotheses are identical to their non-sharpened versions.
  
  \subsection{Choosing useful primes}
   
   In this section, we suppose that (H.$0^\sharp$)--(H.$5^\sharp$) are satisfied, and that the coefficient ring $R$ is Artinian and principal.
   
   \begin{proposition}(cf.~\cite[Proposition 3.6.1]{mazurrubin04})
    \label{prop:goodprimes}
    Let $c_1, c_2 \in H^1(\QQ, T_1)$ and $c_3, c_4 \in H^1(\QQ, T_1^\vee)$ be non-zero elements. For every $k \ge 1$ there is a positive-density set of primes $S \subseteq \cP_k^\sharp$ such that for all $\ell \in S$, the localisations $(c_j)_\ell$ are all non-zero.
   \end{proposition}
   
   \begin{proof}
    We may assume without loss of generality that $\frm^k = 0$. We simply imitate the arguments of \emph{op.cit.} for $T = T_1$, but with the field $F$ replaced by the slightly larger field
    \[ 
     F^\sharp = \QQ\Big(T_1, T_2, \mu_{p^k}\Big).
    \]
    If $\tau \in G_{\QQ}$ is as in (H.$2^\sharp$), then any prime $\ell$ whose Frobenius in $\Gal(F^\sharp / \QQ)$ is conjugate to $\tau$ acts invertibly on $T_2$, so the set of primes $S$ constructed in \emph{op.cit.} is automatically contained in $\cP_k^\sharp$.
   \end{proof}
   
   \begin{remark}
    In fact a slight refinement of this statement is true: if $R = \tilde R / I$ for some larger ring $\tilde R$, and $T_i = \tilde T_i \otimes R$, then we may arrange that $S$ is contained in the set $\tilde \cP_k^\sharp$ defined with the $\tilde T_i$ in place of the $T_i$. Note that $\tilde \cP_k^\sharp$ will in general be smaller than $\cP_k^\sharp$, even in the case $T_2 = \{0\}$. This minor detail appears to have been overlooked in \cite{mazurrubin04}, and this slight strengthening of Proposition 3.6.1 is actually needed for some of the proofs in \emph{op.cit.}, in particular for those dealing with the module $\overline{\mathrm{KS}}(T)$ of generalised Kolyvagin systems.
   \end{remark}

  \subsection{Bounding the Selmer group}
  
   With Proposition \ref{prop:goodprimes} in place of Mazur and Rubin's Proposition 3.6.1, we obtain generalisations of all of the theorems of \cite{mazurrubin04} to this setting. For example, we have the following generalisation of Theorem 5.2.2 (and Remark 5.2.3) of \emph{op.cit.}:
   
   \begin{corollary}
    \label{prop:selbound}
    Suppose $R$ is the ring of integers of a finite extension of $\Qp$, and $(T_1, T_2, \cF, \cP)$ satisfy (H.$0^\sharp$)--(H.$6^\sharp$). Let
    \[ 
     \pmb\kappa \in \overline{\mathrm{KS}}(T, \cF, \cP) 
     \coloneqq \varprojlim_k \left(\varinjlim_j \mathrm{KS}(T / \mathfrak{m}^k T, \cF, \cP \cap \cP_j^\sharp)\right).
    \]
    Then
    \[ 
     \operatorname{length}_R\left(H^1_{\cF^\vee}(\QQ, T^\vee)\right) \le \max \left\{ j : \kappa_1 \in \mathfrak{m}^j H^1(\QQ, T)\right\}.
    \]
   \end{corollary}
   
   The same also applies to Theorem 5.3.10 of \emph{op.cit.}~(bounding a Selmer group over the cyclotomic $\Zp$-extension of $\QQ$) and to Theorem 12.3.5 of \cite{KLZ1b} (allowing a ``Greenberg-style'' local condition at $p$).

\providecommand{\bysame}{\leavevmode\hbox to3em{\hrulefill}\thinspace}
\providecommand{\MR}[1]{\relax}
\renewcommand{\MR}[1]{%
 MR \href{http://www.ams.org/mathscinet-getitem?mr=#1}{#1}.
}
\providecommand{\href}[2]{#2}
\newcommand{\articlehref}[2]{\href{#1}{#2}}

\end{document}